\documentclass{amsart} 
\usepackage{amssymb,times, amscd,amsmath,amsthm, xypic}
\usepackage{geometry}
\author{J\"{o}rg Sch\"{u}rmann and Shoji Yokura}
\address
{J\"{o}rg Sch\"{u}rmann:
Westf. Wilhelms-Universit\"{a}t,
Mathematisches Institut, Einsteinstrasse 62,
48149 M\"{u}nster, Germany }
\email {jschuerm@math.uni-muenster.de}

\address
{Department of Mathematics and Computer Science, 
Faculty of Science, 
Kagoshima University, 21-35 Korimoto 1-chome, Kagoshima 890-0065, Japan}
\email {yokura@sci.kagoshima-u.ac.jp}

\title
{Grothendieck groups and 
a categorification of additive invariants}
\thanks {
\\
(1) J\"{o}rg Sch\"{u}rmann: supported by the SFB 878 ``groups, geometry and actions''.\\
(2) Shoji Yokura: partially supported by Grant-in-Aid for Scientific Research
(No. 21540088), the Ministry of Education, Culture, Sports, Science and Technology (MEXT), Japan 
}

\begin{document} 
\numberwithin{equation}{section}
\newtheorem{thm}[equation]{Theorem}
\newtheorem{pro}[equation]{Proposition}
\newtheorem{prob}[equation]{Problem}
\newtheorem{cor}[equation]{Corollary}
\newtheorem{con}[equation]{Conjecture}
\newtheorem{lem}[equation]{Lemma}
\theoremstyle{definition}
\newtheorem{ex}[equation]{Example}
\newtheorem{defn}[equation]{Definition}
\newtheorem{rem}[equation]{Remark}
\renewcommand{\rmdefault}{ptm}
\def\alp{\alpha}
\def\be{\beta}
\def\jeden{1\hskip-3.5pt1}
\def\om{\omega}
\def\bigstar{\mathbf{\star}}
\def\ep{\epsilon}
\def\vep{\varepsilon}
\def\Om{\Omega}
\def\la{\lambda}
\def\La{\Lambda}
\def\si{\sigma}
\def\Si{\Sigma}
\def\Cal{\mathcal}
\def\m {\mathcal}
\def\ga{\gamma}
\def\Ga{\Gamma}
\def\de{\delta}
\def\De{\Delta}
\def\bF{\mathbb{F}}
\def\bH{\mathbb H}
\def\bPH{\mathbb {PH}}
\def \bB{\mathbb B}
\def \bA{\mathbb A}
\def \bR{\mathbb R}
\def \bC{\mathbb C}
\def \bOB{\mathbb {OB}}
\def \bM{\mathbb M}
\def \bOM{\mathbb {OM}}
\def \mA{\mathcal A}
\def \mB{\mathcal B}
\def \mC{\mathcal C}
\def \mR{\mathcal R}
\def \mH{\mathcal H}
\def \mM{\mathcal M}
\def \mO{\mathcal O}
\def \mTOP{\mathcal {TOP}}
\def \mAB{\mathcal {AB}}
\def \bK{\mathbb K}
\def \bG{\mathbf G}
\def \bL{\mathbf L}
\def\bN{\mathbb N}
\def\bR{\mathbb R}
\def\bP{\mathbb P}
\def\bZ{\mathbb Z}
\def\bC{\mathbb  C}
\def \bQ{\mathbb Q}
\def\op{\operatorname}

\maketitle

\begin{abstract} A topologically-invariant and additive homology class is mostly not a natural transformation as it is. In this paper we discuss turning such a homology class into a natural transformation; i.e., a ``categorification" of it. In a general categorical set-up we introduce \emph{a generalized relative Grothendieck group} from a cospan of functors of categories and also consider \emph{a categorification of additive invariants on objects}. As an example, we obtain a general theory of characteristic homology classes of singular varieties.
\end{abstract}

\section{Introduction}\label{intro}

Characteristic classes are important invariants in modern geometry and topology to investigate other 
invariants. Classically characteristic classes with values in cohomology theory were
considered for real or complex vector bundles. The main feature of them 
is that they are formulated as natural transformations from the 
contravariant functor of vector bundles to the cohomology
theory. When it comes to the case of possibly singular
varieties, characteristic classes are considered in
homology theory, instead of cohomology theory and still
formulated as \emph{natural transformations from a 
covariant functor $\m F$ to a (suitable) homology
theory $H_*$, satisfying the normalization condition that for a smooth variety $X$ the value of a distinguished element $\Delta_X$ of $\m F(X)$ is equal to the Poincar\'e dual of the corresponding characteristic cohomology class of the tangent bundle}:
$$\tau_{c\ell}: \m F(-) \to H_*(-) \quad \text{such that for $X$ smooth} \quad \tau_{c\ell}(\Delta_X) = c\ell(TX) \cap[X].$$
Most important and well-studied theories of
characteristic homology classes of singular varieties are the following ones,
where we consider here for simplicity 
only the category of complex algebraic varieties
and functoriality is 
required only for proper morphisms.
Below, in (1) and (2) the homology theory $H_*(X)$ is either Chow groups $CH_*(X)$ or the even-degree Borel--Moore homology groups $H^{BM}_{2*}(X)$, whereas in (3) $X$ is assumed to be compact and thus the homology theory $H_*(X)$ is the usual even-degree homology group $H_{2*}(X)$:

\begin{enumerate}
\item MacPherson's Chern class \cite{MacPherson} $c^{Mac}_*: F(-) \to H_*(-)$ is the unique natural transformation from the covariant functor $F$ of constructible functions to the homology $H_*$,
satisfying the normalization condition that for a smooth variety $X$ the value of the characteristic function is the Poincar\'e dual of the total Chern class of the tangent bundle: $$c^{Mac}_*(\jeden_X) = c(TX) \cap [X].$$
\item Baum--Fulton--MacPherson's Todd class \cite{BFM} $td^{BFM}_*: G_0(X) \to H_*(X)\otimes \bQ$ is the unique natural transformation from the covariant functor $G_0$ of Grothendieck groups of coherent sheaves to the rationalized homology $H_*\otimes \bQ$,
satisfying the normalization condition that for a smooth variety $X$ the value of (the class of) the structure sheaf is the Poincar\'e dual of the total Todd class of the tangent bundle: $$td^{BFM}_*([\m O_X]) = td(TX) \cap [X].$$
\item Goresky-- MacPherson's homology $L$-class \cite{GM}, which is extended as a natural transformation by  Cappell-Shaneson \cite{CS} (also see \cite{Yokura-TAMS}), $L^{CS}_*: \Omega(X) \to  H_*(X)\otimes \bQ$ is a natural transformation from the covariant functor $\Omega$ of cobordism groups of self-dual constructible sheaf complexes to the rationalized homology theory $H_*\otimes \bQ$,
satisfying the normalization condition that for a smooth variety $X$ the value of the
(class of the) shifted constant sheaf is the Poincar\'e dual of the total Hirzebruch--Thom's $L$-class of the tangent bundle: $$L^{CS}_*([\bQ_X[\op{dim} X]]) = L(TX) \cap [X].$$ Here $X$ is assumed to be compact.\\
\end{enumerate}

Recently these three theories $c^{Mac}_*, td^{BFM}_*$ and $L^{CS}_*$ are ``unified" by the motivic Hirzeruch class  \cite{BSY} (see also \cite{SY}, \cite{Schuermann-MSRI} and \cite{Yokura-MSRI}) ${T_y}_*: K_0(\m V_{\bC}/X) \to H_*(X)\otimes \bQ[y]$, which is the unique natural transformation such that for a smooth variety $X$ the value of the isomorphism class of the identity $X \xrightarrow {id_X} X$  is the Poincar\'e dual of the generalized Hirzebruch--Todd class $T_y(TX)$ of the tangent bundle: $${T_y}_*([X \xrightarrow {id_X} X]) = T_y(TX) \cap [X].$$

Here $K_0(\m V_{\bC}/X)$ is the relative Grothendieck group of the category $\m V_{\bC}$ of complex algebraic varieties, i.e., the free abelian group generated by the isomorphism classes $[V \xrightarrow {h} X]$ of morphism $h \in hom_{\m V_{\bC}}(V, X)$ modulo the relations
\begin{itemize}
\item $[V_1 \xrightarrow {h_1} X] + [V_2 \xrightarrow {h_2} X] =  [V_1 \sqcup V_2 \xrightarrow {h_1 + h_2} X],$ with $\sqcup$ the disjoint union, and
\item  $[V \xrightarrow {h} X] = [V \setminus W \xrightarrow {h_{V \setminus W}} X] +  [ W \xrightarrow {h_{W}}  X]$ for $W \subset V$  a closed subvariety of $V$. 
\end{itemize}
The generalized Hirzebruch--Todd class $T_y(E)$ of the complex vector bundle $E$ (see \cite{Hirzebruch, HBJ}) is defined to be 
$$T_y(E) := \prod _{i=1}^{\op {rank} E} \left (\frac {\alp _i(1+y)}{1-e^{-\alp _i(1+y)}} - \alp _i y \right ) \in H^*(X) \otimes \bQ[y],$$
where $\alp _i$ is the Chern root of $E$, i.e., $\displaystyle c(E) = \prod_{i=1}^{\op{rank} E} (1 + \alp_i).$ Note that 
\begin{enumerate}
\item $T_{-1}(E) =c(E)$ is the Chern class, 
\item $T_{0}(E) =td(E)$ is the Todd class and 
\item $T_{1}(E) =L(E)$ is the Thom--Hirzebruch $L$-class.\\
\end{enumerate}
The ``unification" means the existence of the following commutative diagrams of natural transformations:

$$\xymatrix{
 K_0(\m V_{\bC}/X) \ar[dd]_{{T_{-1}}_*} \ar[dr]^{\, \,const} &&  K_0(\m V_{\bC}/X) \ar[dd]_{{T_{0}}_*}  \ar[dr]^{\, \,coh}  &&  K_0(\m V_{\bC}/X)\ar[dd]_{{T_{1}}_*}  \ar[dr]^{\, \,sd} \\
 & F(X) \ar[dl] ^{c^{Mac}_*\otimes \bQ} && G_0(X) \ar[dl]^{td^{BFM}_*}  && \Omega(X)  \ar[dl]^{L^{CS}_*} \\
 H_*(X)\otimes \bQ, &&  H_*(X)\otimes \bQ, &&  H_*(X)\otimes \bQ. }
$$

Here $const:K_0(\m V_{\bC}/X) \to F(X)$ is defined by $const([V \xrightarrow {h} X]):= h_* \jeden _V$. The other two comparison transformations are characterized by $coh([V \xrightarrow {h} X]) = h_*\m ([O_V])$ and $sd([V \xrightarrow {h} X]) = h_*([\bQ_V[\op{dim}V]])$ for $V$ smooth and $h$ proper.
${T_y}_*(X) := {T_y}_*([X \xrightarrow {id_X} X])$ is called the \emph{motivic Hirzebruch class of $X$}. For more details see \cite{BSY}.\\

Besides these characteristic classes formulated as natural
transformations, there are several important homology
classes which are usually not formulated or not captured as such natural transformations; for example: Chern-Mather class $c_*^M(X)$ (\cite{MacPherson}), Segre--Mather class $s_*^M(X)$ (\cite{Yokura-segre}), Fulton's canonical Chern class $c^F_*(X)$ (\cite{Fulton-book}), Fulton--Johnson's Chern class $c^{FJ}_*(X)$  (\cite{Fulton-Johnson, Fulton-book}), Milnor class $\m M(X)$ (e.g., see \cite{Aluffi}, \cite{BLSS}, \cite{PP}, \cite{Yokura-milnor}, etc.), which is (up to sign) the difference of the MacPherson's Chern class and the Fulton--Johnson's Chern class, etc. 

In \cite{Yokura-JS} we captured Fulton--Johnson's Chern class and the Milnor class $\m M(X)$ as natural transformations, with the latter one as a special case of the Hirzebruch--Milnor class (also see \cite{CMSS}), using the motivic Hirzebruch class.
In this paper we generalize the approach and results of \cite{Yokura-JS} to a more general
abstract categorical context (cf. \cite{Yokura-RIMS}).\\

Let $\m B$ be a category with a \emph{coproduct} $\sqcup$.  Here we assume that a fixed coproduct
$X \sqcup Y$ of any two objects $X,Y\in ob(\m B)$ has been choosen, as well as an 
initial object $\emptyset$ in $\m B$, so that we can view $(\m B,\sqcup)$ as a
symmetric monoidal category with unit $\emptyset$.
Examples are  the category $\mathcal{TOP}$ (resp., $\mathcal {TOP}_{lc}$) of (locally compact) topological spaces, 
the category $\mC^{\infty}$ of $\m C^{\infty}$-manifolds,
or the category $\mathcal V_k$ of algebraic varieties (i.e., reduced separated schemes of finite type) over a base field $k$, with the coproduct $\sqcup$ given by the usual disjoint union. 
Moreover, for $\m B=\mathcal {TOP}_{lc}, \mC^{\infty}$ or $\mathcal V_k$, we often consider them as categories only with respect to proper morphisms.
Let $\mathcal {AB}$ be the category of abelian groups with the coproduct $\sqcup$ given by the direct sum $\oplus$. Then a covariant functor $\m H : \mathcal B \to \mathcal {AB}$ is called  \emph{additive} (cf. \cite{Levine-Morel}), if it preserves the coproduct structure $\sqcup$, i.e., if we get the following isomorphism:
$$(i_X)_*\oplus (i_Y)_*: \m H(X) \oplus \m H(Y) \xrightarrow {\cong} \m H(X \sqcup Y)$$
where $i_X: X \to X \sqcup Y$ and $i_Y:Y \to X \sqcup Y$ are the canonical injections. In particular, $\m H(\emptyset) = \{0\}$. Examples of such an additive functor $\mH$ are a generalized homology theory $H_*$ for $\m B = \m {TOP}$, the Borel--Moore homology theory 
$H_*^{BM}$ for $\m B = \m {TOP}_{lc}$, the functor $G_0$ of Grothendieck groups of coherent sheaves or the  Chow groups $CH_*$ for $\m B = \m V_k$, as well as the functor $F$ of constructible functions or $\Omega$ of cobordism groups of self-dual constructible sheaf complexes for $\m B = \m V_{\bC}$.\\

Let $\mC$ be a category of ``(topological) spaces with some additional structures \emph{stable} under 
$\sqcup$", such as the category of (complex) algebraic varieties, the category of (compact) topological spaces, or the category of oriented smooth (i.e., $C^{\infty}$-) manifolds, etc. 
More abstractly, $\mC$ is a \emph{symmetric monoidal category} $(\mC, \sqcup)$ together with
a strict monoidal forgetful functor
 $\frak f: (\m C, \sqcup) \to (\m B, \sqcup)$. Then an \emph {additive invariant} on objects in $ob(\mC)$ with values in the additive functor $\m H$ is given by an element $\alp(X) \in \mH(\frak f(X))$ for all $X \in ob(\m C)$ satisfying 
$$\alp(X \sqcup Y) = (i_{\frak f(X)})_* \alp(X) + (i_{\frak f(Y)})_*\alp(Y),$$ 
where $i_{\frak f(X)}$ and $i_{\frak f(Y)}$ are the corresponding canonical inclusions.
Note that the sum or difference $\alpha\pm \beta$ of two additive invariants with values in 
$\m H$ is again such an invariant.\\

\begin{ex}\label{singular} 
Let $\m B=\m V_{\bC}$ be the category of complex algebraic varieties and $\m C$ be a full subcategory of $\m B$ which is stable under disjoint union, isomorphisms and contains the initial object $\emptyset$ and the final object $pt$ (e.g., \emph{an admissible subcategory} in the sense of Levine--Morel \cite{Levine-Morel}), such as the category $\m V_{\bC}^{sm}$ of smooth varieties, $\m V^{lpd}_{\bC}$ of locally pure dimensional varieties, $\m V_{\bC}^{emb}$ of varieties embeddable into smooth varieties, or $\m V_{\bC}^{lci}$ of local complete intersections. Then we have the following additive invariants, some of which are already mentioned above:
\begin{enumerate}
\item characteristic homology classes $c\ell_*(X) := c\ell(TX) \cap [X] \in H_*(X)\otimes R$ for $X \in ob(\m V^{sm}_{\bC})$ with $[X]$ the fundamental class and $TX$ the tangent bundle, or the corresponding virtual classes $c\ell_*(X)$ for $X \in ob(\m V^{lci}_{\bC})$, with $TX$ the virtual tangent bundle and $c\ell(TX)$  a characteristic cohomology class.
\item Fulton's canonical Chern class or Fulton--Johnson Chern class $c^F_*(X)$, $c^{FJ}_*(X)$ for $X \in ob(\m V_{\bC}^{emb})$.
\item Mather-type characteristic homology classes $c\ell^{Ma}_*(X)$, such as Chern--Mather class and Segre--Mather class, or the local Euler obstruction $Eu_X \in F(X)$ as a constructible function for $X \in ob(\m V_{\bC}^{lpd})$.
\item the (class of the) self-dual Intersection Homology complex $[\m {IC}_X] \in \Omega (X)$ and for $X$ compact its 
Goresky--MacPherson's $L$-homology class $L^{GM}(X) \in H_*(X)$ for 
$X \in ob(\m V^{lpd}_{\bC})$.
\end{enumerate}
\end{ex}

A \emph{categorification of the additive invariant $\alp$} with values in the additive functor $\m H$ is meant to be an associated \emph{natural transformation}
$$\tau_{\alp}: \m H_{\alp} \to \m H\circ \frak f$$
from a \emph{covariant functor} $\m H_{\alp}$ on the category $\m C $, such that 
$$\tau_{\alp}(\delta_X) = \alp(X)$$
for some \emph{distinguished element} $\delta_X \in \m H_{\alp}(X)$ and all $X \in ob(\m C)$.\\

To construct such a covariant functor $\m H_{\alp}$, we introduce \emph{generalized relative Grothendieck groups}, using comma categories in a more abstract categorical context.
The construction of such a covariant functor is hinted by the definition of the relative Grothendieck group $ K_0(\m V_{\bC}/X)$ and more clearly by the
description of the oriented bordism group $\Omega^{SO}_*(X)$.

The oriented bordism group $\Omega^{SO}_m(X)$ of a topological space $X$
is defined to be the free abelian group generated by the isomorphism classes $[M \xrightarrow h X]$ of \emph{continuous maps} $M \xrightarrow h X$
from \emph{closed oriented smooth manifolds} $M$ of dimension $m$  to the given \emph{topological space} $X$, modulo the following relations
\begin{enumerate}
\item $[M \xrightarrow h X] + [M' \xrightarrow {h'} X] = 
[M \sqcup M' \xrightarrow {h + h'}  X]$,
\item if $M \xrightarrow h X$ and $M' \xrightarrow  {h'} X$ are \emph{bordant}, 
then $[M \xrightarrow h X] = [M' \xrightarrow  {h'} X].$
\end{enumerate}
Note that if $M \xrightarrow h X$ and $M' \xrightarrow  {h'} X$ are isomorphic to each other, i.e., there exists an orientation preserving diffeomorphism $\phi: M \cong M'$ such that the following diagram commutes \emph{as topological spaces}:
$$\xymatrix{M \ar[dr]_ {h}\ar[rr]^ {\phi} && M' \ar[dl]^{h'}\\& X,}$$
then $M \xrightarrow h X$ and $M' \xrightarrow  {h'} X$ are clearly bordant to each other. Hence we can say that the bordism group $\Omega^{SO}_m(X)$
is defined to be the free abelian group generated by isomorphism classes of continuous maps $M \xrightarrow h X$
from \emph{closed oriented smooth manifolds} $M$ of dimension $m$  to the given \emph{topological space} $X$, modulo the  relations (1) and  (2) above. 
The bordism group $\Omega^{SO}_m(X)$ can also be described slightly differently as follows: the set of the bordism classes $[M \xrightarrow h X]$ can be turned into a commutative semi-group or monoid $\mathcal M(X)$ by the relation (1) above, with $0=[\emptyset \to X]$.
Then the Grothendieck group or the group completion of the monoid $\mathcal M(X)$ is nothing but the above bordism group $\Omega^{SO}_m(X)$. \\

In the definition of the bordism group two categories are involved:
\begin{itemize}
\item the category $\mC^{\infty}_{co}$ of closed oriented $\m C^{\infty}$-manifolds, i.e., compact oriented $\m C^{\infty}$-manifolds without boundary,
\item the category $\mTOP$ of topological spaces.
\end{itemize}
It should be emphasized that \emph{even though we consider such a finer category $\mC^{\infty}_{co}$ for a source space $M$ the map $h:M \to X$ of course has to be considered in the crude category $\mTOP$}.

The bordism group $\Omega^{SO}_*$ is a generalized homology theory, in particular $\Omega^{SO}_*$ is a covariant functor 
$$\Omega^{SO}_*: \mTOP \to \mAB,$$
 where $\mAB$ is the category of abelian groups.
 
Clearly we can consider this covariant functor on a different category finer than the category $\mTOP$ of topological spaces, e.g.,  the category $\m V_{\bC}$ of complex algebraic varieties. Namely we consider \emph{continuous} maps $h: M \to V$ from closed oriented manifolds $M$ to a complex algebraic variety $V$. Then we still get a covariant functor
$$\Omega^{SO}_*: \m V_{\bC}  \to \mAB.$$
In this set-up \emph{three different categories} $\mC^{\infty}_{co}$, $\mTOP$ and $\m V_{\bC}$ are involved. More precisely, we have the following forgetful functors $\frak f_s: \mC^{\infty}_{co} \to \mTOP$ and $\frak f_t: \m V_{\bC} \to \mTOP$ (here ``s" and ``t" mean ``source object" and ``target object):
$$\mC^{\infty}_{co} \xrightarrow {\frak f_s} \mTOP \xleftarrow {\frak f_t} \m V_{\bC}.$$
And the commutative triangle 
$$\xymatrix{M \ar[dr]_ {h}\ar[rr]^ {\phi} && M' \ar[dl]^{h'}\\& V}$$
really means
$$\xymatrix{\frak f_s(M) \ar[dr]_ {h}\ar[rr]^ {\frak f_s(\phi)} && \frak f_s (M') \ar[dl]^{h'}\\& \frak f_t(V).}$$

Abstracting this situation, we deal in this paper with a more general situation of a cospan of categories 
$$\mC_s  \xrightarrow {S} \mB \xleftarrow {T} \mC_t,$$
where $\m B$ is a category with (chosen) coproducts $\sqcup$, $(\m C_s, \sqcup)$ is a \emph{symmetric monoidal category} with $S: \mC_s \to \mB$ a strict monoidal functor respecting the units 
$\emptyset$, and $T: \mC_t \to \mB$ just a functor. 

Two triples $(V, X, h), (V', X, h')$, where $V, V'  \in ob(\m C_s), X \in ob(\mC_t)$, $h \in hom_{\m B}(S(V), T(X))$ and $h' \in hom_{\m B}(S(V'), T(X))$, are called \emph{isomorphic}, if there exists an isomorphism $\phi:V \xrightarrow {\cong} V' \in hom_{\mC_s}(V, V')$ such that the following diagram commutes, just like as above:
$$\xymatrix{S(V) \ar[dr]_ {h}\ar[rr]^ {S(\phi)} && S(V') \ar[dl]^{h'}\\& T(X).}$$

Then the isomorphism classes $[(V, X, h)]$ of triples $(V, X, h)$ can be turned into a monoid $\m M(\mC_s \xrightarrow S \mB /T(X))$ by
$$[(V, X,h)] \sqcup [(V', X,h')] := [(V \sqcup V, X, h + h')],$$
with unit $0=[(\emptyset, X,h)]$.
The associated Grothendieck group is denoted by $K(\mC_s \xrightarrow S \mB /T(X))$ (respectively
$K(\mC_s \xrightarrow S \mB /X)$ in case $\mC_t =B$ and $T=\op{id}_{\mC_t}$)
and called the \emph{canonical generalized $(S,T)$-relative Grothendieck group}. It is a covariant functor with respect to $X$ by composing $h$ with $T(f)$ for $f \in hom_{\mC_t}(X, X')$. 

\begin{rem}\label{dist}
In the following two cases, a \emph{distinguished element $\delta_X \in K(\mC_s \xrightarrow S \mB /T(X))$} is available:
\begin{enumerate}
\item[(i)] $\delta_{S(V)} := [(V, S(V), \op{id}_{S(V)})]$ for $X = S(V)$ in the case when $T= \op{id}_{\mC_t} :\mC_t \to \mC_t =\mB$ is the identity functor.
\item[(ii)] $\delta_V:= [(V, V, \op{id}_{S(V)})]$ for $X =V$ in the case when $\mC_t = \mC_s$ and $T = S$.\\
\end{enumerate}
\end{rem}

When it comes to a categorification of an additive invariant $\alp$ with values in an additive covariant functor $\mH: \mB \to \mathcal {AB}$, we sometimes need to consider a stronger notion of an isomorphism of triples $(V, X, h)$ (e.g. in the context with $V$ an ``oriented space"): Let $(V, X, h), (V', X, h')$ be two triples as above. Then they are called \emph{$\alp$-isomorphic} if the following holds: they are isomorphic as above by an isomorphism 
$\phi:V \xrightarrow {\cong} V' \in hom_{\mC_s}(V, V')$ such that
$$(S(\phi))_*\alp(V) = \alp(V').$$

If this last equality holds for any isomorphism $\phi:V \xrightarrow {\cong} V' \in hom_{\mC_s}(V, V')$, then $\alp$ is also called an \emph{isomorphism invariant}.
So in this case there is no difference between $\alp$-isomorphism classes and
isomorphism classes for such triples $(V, X, h)$.
The $\alp$-isomorphism class of $(V, X, h)$ is denoted by $[(V, X, h)]_{\alp}$ and 
these can be turned as before into a monoid $\m M_{\alp}(\mC_s \xrightarrow S \mB /T(X))$,
since $\alp$ is additive.
The associated generalized $(S,T)$-relative Grothendieck group is denoted by $K_{\alp}(\mC_s \xrightarrow S \mB /T(X))$, respectively
$K_{\alp}(\mC_s \xrightarrow S \mB /X)$ in case $\mC_t =B$ and $T=\op{id}_{\mC_t}$.
 If it is clear that we consider this group $K_{\alp}(\mC_s \xrightarrow S \mB /T(X))$ from the context, we sometimes omit the suffix $\alp$ from the notation, e.g. if $\alp$ is  an isomorphism invariant.
Similarly, if we consider  $\m B=\mathcal {TOP}_{lc}, \mC^{\infty}$ or $\mathcal V_k$ only as categories with respect to proper morphisms,
then this is indicated by the notation $K^{prop}_{\alp}(\mC_s \xrightarrow S \mB /T(X)),
K^{prop}_{\alp}(\mC_s \xrightarrow S \mB /X)$ etc.
 
\begin{thm}(A ``categorification" of an additive invariant) \label{cat}
Let $\mH: \mB \to \mathcal {AB}$ be an additive functor on $\mB$  with $T' = \mH \circ T$. Then an additive invariant $\alp$ on $ob(\mC_s)$ with values in $\mH$ induces a natural transformation on $\mC_t$:
$$\tau_{\alp}:K_{\alp}\bigl(\mC_s \xrightarrow S \mB/T(-) \bigr) \to T'(-); \, \tau_{\alp}([(V, X, h)]_{\alp}) := h_*(\alp(V)).$$
Assume $T$ is full, and consider one of the cases (i) or (ii) as in Remark \ref{dist}. Then $\tau_{\alp}$ is the unique natural transformation satisfying
$$\tau_{\alp}([(V, S(V), \op {id}_{S(V)})]_{\alp}) = \alp(V) \,\, \, \text {or} \, \, \, \tau_{\alp}([(V, V, \op {id}_{S(V)})]_{\alp}) = \alp(V).$$
\end{thm}
Let us illustrate this result by some examples.

\begin{ex} 
Let $\mTOP$ be the category of topological spaces and let $\mC^{\infty}_{co}$ be the category of closed oriented $C^{\infty}$-manifolds, whose morphisms are just differentiable maps.
Consider the following cospan of these categories:
$$\mC^{\infty}_{co} \xrightarrow {\frak f} \mTOP \xleftarrow {id_{\mTOP}} \mTOP,$$
where $\frak f: \mC^{\infty}_{co} \to \mTOP$ is the forgetful functor. 
Then the fundamental class $[M]\in H_*(M)$ for $M\in ob(\mC^{\infty}_{co})$ is an
additive invariant with values in the usual homology $H_*$.
More generally, let $c\ell(E)\in H^*(-;R)$ be a contravariant functorial characteristic class of (isomorphism classes of) oriented vector bundles $E$. Then also $$\alp(M):=c\ell(TM)\cap [M]\in H_*(M;R)$$ is  additive,
where we recover the fundamental class for $c\ell$ the unit $1\in H^*(-)$.
Assume $c\ell$ is normalized in the sense that $1=c\ell^0(E)\in H^0(-;R)$, so that
a diffeomorphism $\phi: M \xrightarrow {\cong} M'$ of such oriented manifolds
is an $\alp$-isomorphism if and only if it is orientation preserving.
Then there exists a unique natural transformation
$$\tau_{c\ell}: K_{\alp}(\mC^{\infty}_{co} \xrightarrow {\frak f} \mTOP /-) \to H_*(-;R)$$
such that for a closed oriented $C^{\infty}$-manifold $M$
$$\tau_{c\ell}([(M, \frak f(M), id_{\frak f(M)})]_{\alp}) =c\ell(TM)\cap [M].$$
Let $\Omega^{SO}_*(X)$ be the oriented bordism group of a topological space $X$. Then we have
$$\Omega^{SO}_*(X) \cong K_{\alp}(\mC^{\infty}_{co} \xrightarrow {\frak f} \mTOP /X) / \sim,$$
with  $\sim$ the bordism relation. And $\tau_{c\ell}$ factorizes over $\Omega^{SO}_*(X)$
in the case when $c\ell$ is stable in the sense that $c\ell(E\oplus \bR)=c\ell(E)$ for $\bR$ the trivial
real line bundle.
\end{ex}

\begin{ex} 
Let $\m V_{\bC}$ be the category of complex algebraic varieties, with $\m C$ the full subcategory $\m V_{\bC}^{sm},\m V^{lpd}_{\bC}, \m V_{\bC}^{emb} $ or $\m V_{\bC}^{lci}$ of 
smooth,  locally pure dimensional, embeddable or local complete intersection varieties, with $$\alp(V):=c\ell_*(V)\in H_*(V)\otimes R$$ for $V\in ob(\m C)$ a corresponding additive characteristic homology class as in  Example \ref{singular}. 
Here the homology theory $H_*(X)$ is either Chow groups $CH_*(X)$ or the even-degree Borel--Moore homology groups $H^{BM}_{2*}(X)$.
Then $\alp$ is also an isomorphism invariant.
Consider the following cospan of categories, where we 
consider only proper morphisms:
$$\mC \xrightarrow {\frak f} \m V_{\bC}   \xleftarrow {id_{\m V_{\bC}}} \m V_{\bC},$$
with $\frak f: \mC \to \m V_{\bC}$ being the inclusion functor. 
Then there exists a unique natural transformation
$$\tau_{c\ell}: K^{prop}(\mC \xrightarrow {\frak f}\m V_{\bC} /-) \to H_*(-)\otimes R,$$
such that for all $V\in ob(\mC)$
$$\tau_{c\ell}([(V, \frak f(V), id_{\frak f(V)})]) =c\ell_*(V).$$
Moreover, there is a tautological surjective natural transformation
$$K^{prop}(\mC \xrightarrow {\frak f} \m V_{\bC} /-) \to K_0(\m V_{\bC}/-)$$
to the relative Grothendieck group $K_0(\m V_{\bC}/-)$ of complex algebraic varieties.
And for $\mC=\m V_{\bC}^{sm}$ and 
$$c\ell_*(V)=T_y(TV)\cap [V]=T_{y*}(V)
\quad \text{or} \quad c\ell_*(V)=c(TV)\cap [V]=c_*(V),$$
i.e., the Hirzebruch homology class or MacPherson's Chern class,
the transformations $\tau_{T_{y*}}$ and $\tau_{c_{*}}$ factorize over $K_0(\m V_{\bC}/-)$ by \cite{BSY}.
\end{ex}

In \S 2 we explain the general categorical background and prove Theorem \ref{cat}, whereas in 
\S\S3-4 we apply it to many different geometric situations, e.g., for obtaining Riemann--Roch type theorems.

\section{Generalized relative Grothendieck groups}

\begin{defn}
Let $\mC$ be a symmetric monoidal category equipped with monoidal structure $\oplus$ and the unit object $\emptyset$. The Grothedieck group $K(\mC)$ is defined to be the free abelian group generated by the isomophism classes $[X]$ of objects $X \in ob(\mC)$ modulo the relations
$$[X] + [Y] = [X \oplus Y], \quad \text{, with $0 = [\emptyset]$.}$$
\end{defn}

A functor $\Phi: \mC_1 \to \mC_2$ of two symmetric monoidal categories is a functor which preserves $\oplus$ in the relaxed sense that there are natural transformations:
$$\Phi(A) \oplus_{\mC_2} \Phi(B) \to \Phi(A \oplus_{\mC_1}B),$$
In some usage it requires an isomorphism
$$\Phi(A) \oplus_{\mC_2} \Phi(B) \cong \Phi(A \oplus_{\mC_1}B),$$
in which case it is  called a \emph{strong} monoidal functor. 
Here we also assume that $\Phi$ respects the units.

\begin{ex} 
Let $\m B$ be a category with a (chosen) \emph{coproduct} $\sqcup$ and 
$\m H : \mathcal B \to \mathcal {AB}$ be an additive covariant functor.
Then we have
$$\m H_*(X \sqcup Y) \cong \m H_*(X) \oplus \m H_*(Y) \quad 
\text{, with $\m H_*(\emptyset)=\{0\}$.}$$
 Examples of such an additive functor $\m H$ are a generalized homology theory $H_*$ for $\m B = \m {TOP}$, the Borel--Moore homology theory 
$H_*^{BM}$ for $\m B = \m {TOP}_{lc}$, the functor $G_0$ of Grothendieck groups of coherent sheaves or the  Chow groups $CH_*$ for $\m B = \m V_k$, as well as the functor $F$ of constructible functions or $\Omega$ of cobordism groups of self-dual constructible sheaf complexes for $\m B = \m V_{\bC}$.
\end{ex}

\begin{lem} 
Let $\Phi: \mC_1 \to \mC_2$ be a strong monoidal functor of two  symmetric monoidal categories.
 Then the map
$$\Phi_*:K(\mC_1) \to K(\mC_2), \quad \Phi_*([X]):= [\Phi(X)]$$
is a well-defined group homomorphism. Namely, the Grothendieck group $K$ is a covariant functor from the category of such categories and functors to the category of abelian groups.
\end{lem}

Now we recall the notion of \emph{comma category} and \emph{fiber category} (e.g., see \cite{Maclane}):
\begin{defn}Let $\mC_s$ (the suffix ``s" meaning \emph{source}), $\mC_t$ (the suffix ``t" meaning \emph{target}) and $\mB$ be categories, and let
$$\mC_s \xrightarrow {S}  \mB \xleftarrow {T} \mC_t$$
be functors among them (which shall be called \emph{a cospan of categories}). The \emph{comma category}, denoted by $(S \downarrow T)$, is formed by
\begin{itemize}
\item the objects of $(S \downarrow T)$ are triples $(V, X, h)$ with $V \in ob(\mC_s)$, $X \in ob(\mC_t)$ and $h \in hom_{\mB}(S(V), T(X))$,
\item the morphisms from $(V, X, h)$ to $(V', X', h')$ are pairs $(g_s,g_t)$ where  $g_s: V \to V'$  is a morphism in $\mC_s$ and  $g_t: X \to X'$  is a morphism in $\mC_t$ such that the following diagram commutes in the category $\mB$:
$$\CD
S(V)   @> S(g_s)>> S(V')  \\
@V h VV @VV h' V\\
T(X) @>>  T(g_t) > T(X'). \endCD
$$ 
\end{itemize}
\end{defn}

\begin{defn} Let $\mC_s \xrightarrow {S}  \mB \xleftarrow {T} \mC_t$ be a cospan and let $(S \downarrow T)$ be the above comma category associated to the cospan. We define the canonical \emph{projection functors} as follows:
\begin{enumerate}

\item $\pi_t: (S \downarrow T) \to \mC_t$ is defined by
\begin{itemize}
\item for an object $(V, X, h)$, $\pi_t((V, X, h)) := X,$
\item for a morphism $(g_s, g_t): (V, X, h) \to (V', X', h')$, $\pi_t((g_s, g_t)) := g_t.$
\end{itemize}

\item $\pi_s: (S \downarrow T) \to \mC_s$ is defined by
\begin{itemize}
\item for an object $(V, X, h)$, $\pi_s((V, X, h)) := V,$
\item for a morphism $(g_s, g_t): (V, X, h) \to (V', X', h')$, $\pi_s((g_s, g_t)) := g_s.$
\end{itemize}
\end{enumerate}
\end{defn}
Namely \emph{a cospan of categories}\,\, $\mC_s \xrightarrow {S}  \mB \xleftarrow {T} \mC_t$ induces \emph{a span of categories} \, \, $\mC_s \xleftarrow {\pi_s}  (S \downarrow T)  \xrightarrow {\pi_t} \mC_t.$

\begin{defn} Let $F: \mC \to \m{D}$ be a functor of two categories. Then, for an object $B \in ob(\m{D})$ \emph{the fiber category} of $F$ over $B$, denoted by $F^{-1}(B)$,  is defined to be the category consisting of
\begin{itemize}
\item the objects $X \in ob(\mC)$ such that $F(X) = B$,
\item for such objects $X, X'$, morphisms $f:X \to X'$ such that $F(f) = id_B$.
\end{itemize}
In  other words, more precisely, this category should be denoted by $F^{-1}(B, id_B)$.
\end{defn}

\begin{ex} As above, let us consider a cospan of categories and its associated span of categories:
$$\mC_s \xrightarrow {S}  \mB \xleftarrow {T} \mC_t, \quad \quad \mC_s \xleftarrow {\pi_s}  (S \downarrow T)  \xrightarrow {\pi_t} \mC_t.$$
\begin{enumerate}
\item For an object $X \in \mC_t$, the fiber category $\pi_t^{-1}(X)$ is nothing but the \emph{S-over category} $(S \downarrow T(X))$, whose objects are \emph{objects $S$-over $T(X)$}, i.e., the triple $(V, X, h)$, and for two triples $(V, X, h)$ and $(V', X, h')$ a morphism from $(V, X, h)$ to $(V', X, h')$ is $g_s \in hom_{\mC_s}(V, V')$ such that the following triangle commutes:
$$\xymatrix{S(V) \ar[dr]_ {h}\ar[rr]^ {S(g_s)} && S(V') \ar[dl]^{h'}\\& T(X).}$$

\item Furthermore, if $\mC_s = \mB$ and $S = id_S$ is the identity functor, then the above \emph{S-over category} $(S \downarrow X)$ is the standard \emph{over category} $(\mB \downarrow X)$, whose objects are objects over $X$, i.e., morphisms $h:V \to X$, and for two tmorphisms $h:V \to X$ and $h:V' \to X$  a morphism from $h:V \to X$ to  $h:V' \to X$ is $g \in hom_{\mB}(V, V')$ such that the following triangle commutes:
$$\xymatrix{V \ar[dr]_ {h}\ar[rr]^ {g} && V' \ar[dl]^{h'}\\& X.}$$
\end{enumerate}
\end{ex}

\begin{rem} For an object $V \in \mC_s$, the fiber category $\pi_s^{-1}(S)$ is nothing but the \emph{T-under category} $(S(V) \downarrow T)$, whose objects are \emph{objects $T$-under $S(V)$}, i.e., the triple $(V, X, h)$, and for two triples $(V, X, h)$ and $(V, X', h')$ a morphism from $(V, X, h)$ to $(V, X', h')$ is $g_t \in hom_{\mC_t}(X, X')$ such that the following triangle commutes:
$$\xymatrix{&S(V) \ar[dl]_ {h}\ar[dr]^ {h'}&\\
T(X) \ar[rr]_{T(g_t)} && T(X').}$$

Similarly, we can think of the \emph{T-under category} $(V \downarrow T)$ and the \emph{under category} $(V \downarrow \mB)$, but we do not write them down here, since we do not use them below in the rest of the paper.
\end{rem}

\begin{pro} Let $\mC_s \xrightarrow {S}  \mB \xleftarrow {T} \mC_t$ be a cospan of categories. Then a morphism $f \in hom_{\mC_t} (X_1, X_2)$ gives rise to the functor between the corresponding fiber categories:
$$T(f)_*: \pi_t^{-1}(X_1)  \to \pi_t^{-1}(X_2) ,$$
which is defined by
\begin{enumerate}
\item For an object $(V, X_1, h)$, $T(f)_*((V, X_1, h)) := (V, X_2, T(f) \circ h).$
\item For a morphism $(g_s, id_{X_1}): (V, X_1, h) \to (V', X_1, h')$ with $g_s \in hom_{\mC_s}(V, V')$, 
$$T(f)_*((g_s, id_{X_1})) := (g_s, id_{X_2}): (V, X_2, T(f) \circ h) \to (V', X_2, T(f) \circ h').$$
$$\xymatrix{S(V) \ar[dr]^ {h}\ar[rr]^ {g}\ar[dddr]_{T(f) \circ h} && S(V') \ar[dl]_{h'} \ar [dddl]^{T(f) \circ h'} \\
& T(X_1) \ar[dd]^(.25){T(f)}\\
&\\
&T(X_2)}$$
\end{enumerate}
\end{pro}

\begin{lem}\label{relK} Let $\m B$ be a category with a \emph{coproduct} $\sqcup$,
with the coproduct of any two objects and an initial object chosen, and let
$(\m C_s, \sqcup)$ be a symmetric monoidal category.
 Let $\mC_s \xrightarrow {S}  \mB \xleftarrow {T} \mC_t$ be a cospan of categories with $S: (\mC_s , \sqcup)\to (\mB, \sqcup)$ a strict monoidal functor and $T:\mC_t \to \mB$ just a functor.  Then for each object $X \in ob(\mC_t)$, the fiber category $\pi_t^{-1}(X)$, i.e. the $S$-over category $(S \downarrow T(X))$ becomes also a symmetric monoidal category with
$$(V, X, h) \sqcup (V', X, h') := (V \sqcup V', X, h + h').$$
\end{lem} 

\begin{cor} Let the situation be as above. A morphism $f \in hom_{\mC_t}(X_1, X_2)$ gives rise to the canonical group homomorphism
$$T(f)_* : K(\pi_t^{-1}(X_1) ) \to K(\pi_t^{-1}(X_2)),$$
and
$$K(\pi_t^{-1}(-)) : \mC_t \to \mAB$$
is a covariant functor from the category $\mC_t$ to the category of abelian groups.
\end{cor}

After these simple observations, we can introduce the following notions:

\begin{defn}[\emph{Generalized relative Grothendieck groups with respect to a cospan of categories}] $ $ \\ 
Let $\mC_s \xrightarrow {S}  \mB \xleftarrow {T} \mC_t$ be a cospan of categories
as in Lemma \ref{relK}.
\begin{enumerate}
\item The Grothendieck group of the fiber category of the projection functor $\pi_t: (S \downarrow T) \to \mC_t$ is denoted by
$$K(\mC_s \xrightarrow {S} \mB /T(X)): = K(\pi_t^{-1}(X)))$$
and called the \emph{generalized $(S,T)$-relative Grothendieck group of $X\in \mC_t$}.
\item If $\mC_t = \mB$ and $T = id_{\mB}$, then $K(\mC_s \xrightarrow {S} \mB /T(X))$ is simply denoted by $K(\mC_s \xrightarrow {S} \mB /X)$.
\item If $S = T= id_{\mC_s} :\mC_s \to \mC_s$ is the identity functor, then the above $id_{\mC}$-relative Grothendieck group $K(\mC_s \xrightarrow {id_{\mC_s}}\mC _s/X)$ is simply denoted by
$K(\mC_s/X)$
and called the \emph{relative Grothendieck group of $X$} or the \emph{relative Grothendieck group of $\m C_s$ over $X$}.
\end{enumerate}
All these relative Grothendieck groups are covariant functors from $\mC_t$ to $\mAB$.
\end{defn}

\begin{rem} If $T(X)=pt$ is a terminal object in the category $\m B$, then all the above relative Grothendieck groups are isomophic to the Grothendieck group $K(\mC_s)$ of the 
symmetric monoidal category $(\mC_s,\sqcup)$:
$$K(\mC_s \xrightarrow {S} \mB /T(X)=pt) \cong K(\mC_s \xrightarrow {S} \mB /pt) \cong K(\mC_s/pt) \cong K(\mC_s).$$
\end{rem}

\begin{pro}\label{comma} Let $\mC_s \xrightarrow {S}  \mB \xleftarrow {T} \mC_t$ be a cospan of categories as in Lemma \ref{relK}, 
$\Phi_b:\mB \to \mB'$ be a covariant functor preserving (chosen) coproducts and set $S':= \Phi_b \circ S$ and $T':=\Phi_b \circ T$. Then we get the canonical natural transformation from the functor $K(\mC_s \xrightarrow {S} \mB /T(-)): \mC_t \to \mAB$ to the functor $K(\mC_s \xrightarrow {S'} \mB' /T'(-)): \mC_t \to \mAB$:

$$\Phi_*: K(\mC_s \xrightarrow {S} \mB /T(-)) \to K(\mC_s \xrightarrow {S'} \mB' /T'(-)),$$
i.e., for a morphism $f \in hom_{\mC_t}(X_1, X_2)$ the following diagram commutes  in the category $\mAB$:
$$\CD
K(\mC_s \xrightarrow {S} \mB /T(X_1))@> \Phi_* >> K(\mC_s \xrightarrow {S'} \mB' /T'(X_1)) \\
@V T(f)_* VV @VV T'(f)_*V \\
K(\mC_s \xrightarrow {S} \mB /T(X_2)) @>> \Phi_*  >K(\mC_s \xrightarrow {S'} \mB' /T'(X_2)). \endCD
$$ 
Here $\Phi_*: K(\mC_s \xrightarrow {S} \mB /T(X)) \to K(\mC_s \xrightarrow {S'} \mB' /T'(X))$ is defined by
$$\Phi_*([(V, X, h)]:= [(V, X, \Phi_b(h))].$$
\end{pro}

\begin{proof} It suffices to show the commutativity $f_* \Phi_* = \Phi_* f_*$ in the square diagram above. First we observe the following morphisms:
$$\xymatrix{\hspace{4cm} S'(V) = \Phi_b(S(V)) \ar[d]^ {\Phi_b(h)} \hspace{2cm}\\
\hspace{4.2cm}  T'(X_1) = \Phi_b(T(X_1)) \ar[d]^{T'(f)= \Phi_b(T(f))} \hspace{2cm}\\
\hspace{2.2cm} T'(X_2) =\Phi_b(T(X_2))} \hspace{2cm}$$

Then we get for $[(V, X_1, h)] \in K(\mC_s \xrightarrow {S} \mB /T(X_1))$:
\begin{align*}
T'(f)_*\Phi_*([(V, X_1, h)]) & = T'(f)_*([(V, X_1, \Phi_b(h))] \\
 & = [(V, X_2, T'(f)\circ \Phi_b(h))] \\
& = [(V, X_2, \Phi_b(T(f) )\circ \Phi_b(h))] \\
& = [(V, X_2, \Phi_b(T(f) \circ h))] \\
& = \Phi_* ([(V, X_2, T(f) \circ h)]) \\
& = \Phi_* T(f)_* ([(V, X_1,h)]). \\
\end{align*}
\end{proof}

As explained in the introduction, when it comes to a categorification of an addtitive invariant $\alp$ with values in an additive covariant functor $\m H$, we need to consider \emph{$\alp$-isomorphism class $[(V,X,h)]_{\alp}$}. Then all the previous results hold even if we replace $[(V,X,h)]$ and $K(\mC_s \xrightarrow {S} \mB /T(-))$ with $[(V,X,h)]_{\alp}$ and $K/{\alp}(\mC_s \xrightarrow {S} \mB /T(-))$, respectively. And we get the following result:
 
\begin{thm}(A ``categorification" of an additive invariant) \label{main}
Let $\m B$ be a category with a \emph{coproduct} $\sqcup$,
with the coproduct of any two objects and an initial object chosen, and let
$(\m C_s, \sqcup)$ be a symmetric monoidal category. Consider a cospan of categories
 $\mC_s \xrightarrow {S}  \mB \xleftarrow {T} \mC_t$,  with $S: (\mC_s , \sqcup)\to (\mB, \sqcup)$ a strict monoidal functor and $T:\mC_t \to \mB$ just a functor.
Let $\mH: \mB \to \mathcal {AB}$ be an additive functor on $\mB$  with $T' = \mH \circ T$. Then an additive invariant $\alp$ on $Obj(\mC_s)$ with values in $\mH$ induces a natural transformation on $\mC_t$:
$$\tau_{\alp}:K_{\alp}(\mC_s \xrightarrow S \mB/T(-)) \to T'(-); \, \tau_{\alp}([(V, X, h)]_{\alp}) := h_*(\alp(V)).$$
Assume $T$ is full, and consider one of the cases (i) or (ii) as in Remark \ref{dist}. Then
$\tau_{\alp}$ is the unique natural transformation satisfying
$$\tau_{\alp}([V, S(V), \op {id}_{S(V)}]_{\alp}) = \alp(V) \,\, \, \text {or} \, \, \, \tau_{\alp}([V, V, \op {id}_{S(V)}]_{\alp}) = \alp(V).$$
\end{thm}

\begin{proof} Note that $\tau_{\alp}$ is well-defined, since we consider $\alp$-isomorphism
classes so that it does not depend on the chosen representative of the class $[(V,X,h)]_{\alp}$. Then $\tau_{\alp}$ becomes a group homomorphism by the definition of an additive invariant and the functoriality of $\m H$:
\begin{align*}
\tau_{\alp}\Bigl([(V, X, h)]_{\alp} + [(V', X, h')]_{\alp}\Bigr) & = \tau_{\alp}([(V, X, h)]_{\alp} \sqcup  [(V', X, h')]_{\alp}) \\
& = \tau_{\alp}([(V \sqcup V', X, h + h')]_{\alp}) \\
& = (h+h')_*(\alp(V \sqcup V'))  \\
& = (h+h')_*\Bigl( \left(i_{\frak f(V)} \right)_*\alp(V) +  \left(i_{\frak f(V')} \right)_*\alp(V') \Bigr)\\
& = (h+h')_*(i_{\frak f(V)})_* (\alp(V)) +  (h+h')_*(i_{\frak f(V')})_*(\alp(V'))\\
& = h_*(\alp(V)) + h'_*(\alp(V')) \\
& = \tau_{\alp}([(V, X, h)]_{\alp})  + \tau_{\alp}([(V', X, h')]_{\alp}),
\end{align*}
since $(h+h') \circ i_{\frak f(V)} =h$ and $(h+h') \circ i_{\frak f(V')} =h'$.
Finally assume $T$ is full so that there is a morphism $f\in hom_{\mC_t}(S(V),X)$ in the case of (i), or
$f\in hom_{\mC_t}(V,X)$ in the case of (ii), with $T(f)=h: S(V)\to T(X)$. Then
$$f_*([(V, S(V), id_{S(V)})]_{\alp})=([(V, X, h)]_{\alp}) \quad \text{or} \quad
f_*([(V, V, id_{S(V)})]_{\alp})=([(V, X, h)]_{\alp})\:,$$
which implies the uniqueness statement.
\end{proof}

\section {A categorification of an additive homology class}
In the rest of the paper we deal with $\mB$ a category of spaces (with some possible extra structures), with the coproduct $\sqcup$ given by the usual disjoint union. 
Examples are  the category  $\mathcal {TOP}_{(lc)}$ of (locally compact) topological spaces, 
the category $\mC_{(co)}^{\infty}$ of (closed oriented) $\m C^{\infty}$-manifolds,
or the category $\mathcal V_k$ of algebraic varieties (i.e. reduced separated schemes of finite type) over a base field $k$.
And the additive covariant functor $\m H : \mathcal B \to \mathcal {AB}$ is most of the time a suitable homology theory,
like usual homology $H_*(-;R)$ with coefficients in some commutative ring $R$ or a generalized homomolgy $H_*$, in case $\mB= \mathcal {TOP}$. 
Moreover, for $\m B=\mathcal {TOP}_{lc}, \mC^{\infty}$ or $\mathcal V_k$, we often consider them as categories only with respect to proper morphisms,
with $\m H$ the (even degree) Borel-Moore homology $H^{BM}_{(2)*}(-;R)$  with coefficients in some commutative ring $R$ for $\m B=\mathcal {TOP}_{lc}, \mC^{\infty}$, or $\m H=CH_*$ (resp. $CH_*\otimes R$) the Chow groups (with coefficients in $R$) for $\m B=\mathcal V_k$.
Similarly, if we further restrict ourselfes 
to projective morphisms in the algebraic context, 
then $\m H : \mathcal V_k \to \mathcal {AB}$  could also be a suitable Borel--Moore functor in the sense of \cite{Levine-Morel}.\\

The corresponding generalized relative Grothendieck groups only with respect to  proper morphisms in $\mB$ 
(and $\alp$ a corresponding additive invariant) are then denoted by
$$K^{prop}_{(\alp)}(\mC_s \xrightarrow {S} \mB /T(-)), K^{prop}_{(\alp)}(\mC_s \xrightarrow {S} \mB /-) \quad \text{and} \quad K^{prop}_{(\alp)}(\mC_s /-).$$
Then one has a tautological group homomorphism (by just forgetting the properness condition)
$$ K^{prop}(\mC_s \xrightarrow {S} \mB /T(X)) \to K(\mC_s \xrightarrow {S} \mB /T(X)),$$
whose image is the subgroup of the generalized relative Grothendieck group $K(\mC_s \xrightarrow {S} \mB /T(X))$ generated by isomorphism classes
$[(V, X, h)]$
with $h:S(V) \to T(X)$ being a \emph {proper} map. 
If we assume that $S(V)$ is compact (resp. complete in the algebraic context) for every $V\in ob(\mC_s)$, then we have
$$K^{prop}(\mC_s \xrightarrow {S} \mB /T(X)) = K(\mC_s \xrightarrow {S} \mB /T(X)).$$

Now our base category $\m B$ of spaces is not only a symmetric monoidal category with respect to the disjoint union $\sqcup$ (with unit
the initial empty space $\emptyset$), but also with respect to the product of spaces $\times$ (with unit the terminal point space $\{pt\}$,
given by $Spec(k)$ in the algebraic context). Moreover, these structures are compatible in the sense that
$$(X\sqcup X')\times Y = (X\times  Y) \sqcup (X'\times Y) \quad \text{and} \quad Y\times (X\sqcup X') = (Y\times X) \sqcup (Y\times X'),$$
with $\emptyset \times Y =\emptyset = Y\times \emptyset$. Similarly the class of proper (or projective) morphisms in $\m B$ (in the algebraic context)
is stable under products $\times$. And we want to discuss the multiplicativity properties of our transformations
$$\tau_{\alp}: K^{(prop)}_{\alp}(\mC_s \xrightarrow {S} \m B /-) \to \m H$$
associated to an additive invariant on objects in $ob(\mC_s)$ with values in a suitable additive functor $\m H : \mathcal B \to \mathcal {AB}$ on the category $\mB$, which may be functorial only with respect to proper (or projective) morphisms. Here we consider for simplicity only the most important case
that $\mC_t =B$ and $T=\op{id}_{\mC_t}$.
Then it is easy see the following:

\begin{pro} \label{mult}
\begin{enumerate}
\item Assume that $\mC_s$ is also a symmetric monoidal category with respect to a product $\times$, such that $S: \mC=\mC_s \to \m B$ is  strict monoidal with respect to $\sqcup$ as well as $\times$ (e.g. $S$ is the inclusion of a subcategory stable under $\sqcup$ and $\times$).
Then $K^{(prop)}(\mC \xrightarrow {S} \m B /X)$  has a functorial bilinear cross product structure:
$$ \times : K^{(prop)}(\mC\xrightarrow {S} \mB /X)  \times K^{(prop)}(\mC \xrightarrow {S} \mB /Y)\to K^{(prop)}(\mC \xrightarrow {S} \mB /X \times Y);$$
$$[(V, X, h)] \times [(W, Y, k)] := [(V \times W, X \times Y,  h \times k)],$$
with $\times \circ (f_*\times  g_*) = (f\times g)_*\circ \times$ for all (proper or projective) morphisms $f,g$ in $\mB$.
\item Assume that the additive functor $\m H : \mB \to \mathcal {AB}$ is endowed with a bilinear cross poduct
$$\boxtimes: \mH(X)\times \mH(Y) \to \mH(X\times Y)$$
such that $\boxtimes \circ (f_*\times  g_*) = (f\times g)_*\circ \boxtimes$ for all (proper or projective) morphisms $f,g$ in $\mB$.
Consider an additive invariant $\alp$ on objects in $ob(\mC)$ with values in  $\m H$, which is \emph{multiplicative} in the sense that
$$\alp(V \times V') = \alp(V)\boxtimes \alp(V')\quad \text{for all $V,V'\in ob(\mC)$}.$$
Then $K^{(prop)}_{\alp}(\mC \xrightarrow {S} \mB /-)$ also gets 
a functorial bilinear cross product structure in the same way as before
and the associated natural transformation $\tau_{\alp}: K^{(prop)}_{\alp}(\mC\xrightarrow {S} \mB /-) \to \mH$ 
given in Theorem \ref{main}
commutes with the cross product, i.e., the following diagram commutes:
$$\begin{CD}
K^{(prop)}_{\alp}(\mC \xrightarrow {S} \mB/X) \times  K^{(prop)}_{\alp}(\mC \xrightarrow {S} \mB /Y ) @> \times  >> 
K^{(prop)}_{\alp}(\mC \xrightarrow {S} \mB /X \times Y )  \\
@V \tau_{\alp} \times \tau_{\alp} VV  @VV  \tau_{\alp}  V \\
\mH(X) \times \mH(Y) @>> \boxtimes > \mH(X \times Y)  \:.
\end{CD} 
$$
\end{enumerate}
\end{pro}

\begin{proof} We only need to prove the commutativity of the last diagram, which follows from
\begin{align*}
\tau_{\alp} \left([(V, X, h)]_{\alp} \times [(W, Y, k)]_{\alp}\right) &=
\tau_{\alp} \left([(V \times W, X \times Y,  h \times k)]_{\alp}\right)\\
&= ( h \times k)_*\left(\alp(V \times W)\right)\\
&= ( h \times k)_*\left(\alp(V) \boxtimes \alp(W)\right)\\
&= h_*(\alp(V))\boxtimes k_*(\alp(W))\\
&= \tau_{\alp} \left([(V, X, h)]_{\alp}\right)  \boxtimes 
\tau_{\alp} \left([(W, Y, k)]_{\alp}\right).
\end{align*}
\end{proof}

Let us illustrate this result in some examples. First we consider the differential-topological context with $\m B = \mathcal {TOP}_{lc}$ the category of locally compact topological spaces, 
and $\mC$ the category $\mC_{(o)}^{\infty}$ or $\mC^{\infty}_{\bC}$ of all differentiable (oriented) or stable complex $\m C^{\infty}$-manifolds, with $S: \mC\to \mB$ the forget functor.
Of course $S$ commutes with $\sqcup$ and $\times$ for $\mC=\mC^{\infty}$ the category
of $\m C^{\infty}$-manifolds. Moreover, any such manifold $V$ has a fundamental class 
$[V]\in H_*^{BM}(V;\bZ_2)$ in Borel-Moore homology with $\bZ_2$-coefficients.
And this fundamental class is additive and multiplicative:
$$[V\sqcup V']= [V]+[V'] \quad \text{and} \quad [V\times V'].$$
When it comes to an oriented (or a stable complex) $\m C^{\infty}$-manifold $V$, then this has
a fundamental class $$[V]\in H_*^{BM}(V;\bZ)$$ in Borel-Moore homology with $\bZ$-coefficients.
And we view them as a category with the same morphisms as for the underlying 
$\m C^{\infty}$-manifolds, i.e. with $\m C^{\infty}$-maps between them (so that diffeomorphisms are the isomorphisms).
Then the disjoint union $V\sqcup V'$ or product $V\times V'$ of two oriented (or   stable complex) $\m C^{\infty}$-manifolds $V, V'$ can also be oriented (or given the structure of a
stable complex) $\m C^{\infty}$-manifold. And there is a natural choice for this so that
the fundamental class $[-]$ becomes  additive and multiplicative as before.
In this way we also get the symmetric monoidal structures $\sqcup$ and $\times$ on the category
of oriented (or a stable complex) $\m C^{\infty}$-manifolds, with the forget functor $S$ to
$\mathcal {TOP}_{lc}$ (or also to $\m C^{\infty}$) commuting with these structures.
Also note that 
$$s_*[V\times V']= (-1)^{dim(V)\cdot dim(V')} [V'\times V]
\in H^{BM}_{dim(V)+dim(V')}(V'\times V;\bZ)$$
for two \emph{connected} oriented manifolds $V,V'$ and $s: V\times V'\stackrel{\sim}{\to} V'\times V$ the symmetry isomorphism, since $H^{BM}_{*}(-;\bZ)$ is \emph{graded-commutative} with respect to the usual cross product $\boxtimes$.

\begin{cor}\label{smooth-gen}
Let $\mC$  be the category $\mC^{\infty}_{(o)}$ or $\mC^{\infty}_{\bC}$ of all (oriented) or stable complex smooth manifolds.
Consider a contravariant functorial characteristic class
$$c\ell(E) \in  H^*(-; \bZ_2) \quad \text{or} \quad c\ell(E) \in  H^{2*}(-;R)$$ 
of (isomorphism classes of) real (oriented) or complex vector bundles, which is
\emph{multiplicative and normalized}, i.e.: 
$$c\ell(E \oplus F) = c\ell(E) \cup c\ell(F)\quad \text{and} \quad 1=c\ell^0(E)\in H^0(-;R).$$ 
For a smooth
(oriented or stable complex) manifold $V$, let 
$$\alp(V) := c\ell(TV) \cap [V]\in H_*^{BM}(-;R).$$
Then the invariant $\alp$ is \emph{additive and multiplicative}. 
By Theorem \ref{main},
there exists a unique natural transformation
$$\tau_{c\ell}: K^{prop}_{\alp}(\mC \xrightarrow {S} \mTOP_{lc} /-) \to H^{BM}_*(-;R)$$
such that for a smooth (oriented or stable complex) manifold $V$
$$\tau_{c\ell}([(V, V, id_{V})]_{\alp}) = c\ell(TV) \cap [V].$$
And $\tau_{c\ell}$ is also multiplicative (by Proposition \ref{mult}), i.e.:
$$\tau_{c\ell}([(V, X, h)]_{\alp} \times [(W, Y, k)]_{\alp}) = \tau_{c\ell}([(V, X, h)]_{\alp}) \boxtimes \tau_{c\ell}([(W, Y, k)]_{\alp}).$$
\end{cor}

\begin{proof}
The invariant $\alp$ is additive by the functoriality of $c\ell$ and the projection formula
for the inclusions $i: V\to V\sqcup V'$ and $i': V'\to  V\sqcup V'$:
\begin{align*}
\alp(V\sqcup V')&= c\ell(T(V\sqcup V'))\cap [V\sqcup V']\\
&= c\ell(T(V\sqcup V'))\cap (i_*[V]+ i'_*[V'])\\
&=i_*\left(i^*c\ell(T(V\sqcup V'))\cap[V]\right) + 
i'_*\left(i'^*c\ell(T(V\sqcup V'))\cap[V']\right)\\
&= i_*\left(c\ell(TV)\cap[V]\right) + 
i'_*\left(c\ell(T V')\cap[V']\right)\\
&= i_*\alp(V)+ i'_*\alp(V').
\end{align*}
The invariant $\alp$ is multiplicative by the functoriality and multiplicativity of $c\ell$:
\begin{align*}
\alp(V\times  V')&= c\ell(T(V\times V'))\cap [V\times V']\\
&=\left( c\ell(TV)\boxtimes c\ell(TV')\right) \cap \left([V]\boxtimes [V']\right)\\
&=\left( c\ell(TV)\cap [V]\right) \boxtimes \left( c\ell(TV')\cap [V']\right)\\
&=\alp(V)\boxtimes \alp(V').
\end{align*}
Note that in the third equality there is no sign appearing, since we only consider
even degree characteristic classes $c\ell$ in the context of oriented or stable complex 
manifolds.
\end{proof}

\begin{rem} \label{or-preserving}
In the context of oriented manifolds in the above corollary, $(V, X,h)$ and $(V', X, h)$ are $\alp$-isomorphic if and only if the
isomorphism $\phi: V \to V'$ (with $h=h'\circ S(\phi)$) is  orientation preserving. Indeed, if $\phi: V \to V'$ is orientation preserving,
 then $S(\phi)_*[V] = [V']$ and $\phi^*TV' \simeq TV$ as oriented vector bundles. Hence we have
\begin{align*}
S(\phi)_*(\alp(V)) & = S(\phi)_*(c\ell(TV) \cap [V]) \\
& = S(\phi)_*(S(\phi)^*c\ell(TV') \cap [V])\\
& = c\ell(TV') \cap S(\phi)_*[V] \,\, \, \,\text {(by the projection formula)}\\
& = c\ell(TV') \cap [V'] = \alp(V').
\end{align*}
Hence $(V, X,h)$ and $(V', X, h)$ are $\alp$-isomorphic. Conversely, if  $(V, X,h)$ and $(V', X, h)$ are $\alp$-isomorphic, then we have $S(\phi)_*(\alp(V)) = \alp(V')$, i.e., $S(\phi)_*(c\ell(TV) \cap [V]) = c\ell(TV') \cap [V']$. Since $c\ell$ is normalized, we have $c\ell = 1 + \cdots$, hence $c\ell(TV) \cap [V] = [V] + \text {lower dimensional classes}$. Therefore we  have
$$S(\phi)_*[V] = [V'].$$
Hence $\phi:V \to V'$ is orientation preserving.

Similarly $\alp$ is an isomorphism invariant in the context of unoriented manifolds. 
In the context of stable complex manifolds, $\alp$ is at least 
invariant under a diffeomorphism $\phi: V\to V'$ of stable complex manifolds who preserves the stable almost complex structure (and therefore also the orientation) in the sense that $\phi^*TV' \simeq TV$ as stable complex vector bundles.
\end{rem}

If we consider in Corollary \ref{smooth-gen} all \emph{compact} (oriented) or stable complex smooth manifolds,
then we get similar results for any \emph{generalized homology theory} $\m H_*$, which has a corresponding fundamental class
$[V]\in \m H_*(V)$ for a compact (oriented) or stable complex smooth manifold $V$, e.g. for a complex oriented (co)homology theory
and $V$ a stable complex smooth manifold.\\

Let us now switch to some counterparts in the algebraic geometric context, with $\mB=\m V_k$  the category of algebraic varieties
(i.e. reduced separated schemes of finite type) over a base field $k$, and $S: \mC\to \m V_k$ the inclusion functor of a (full) subcategory
$\mC$ stable under isomorphisms, disjoint union $\sqcup$ and product $\times$, with $\emptyset, Spec(k)\in ob(\mC)$.
First we consider the subcategory $\mC= \m V_k^{sm}$ of smooth varieties. The proof of the following result is identical to that of
Corollary \ref{smooth-gen}.

\begin{cor} \label{cor-alg}
Consider the cospan of categories
$$\m V_k^{sm} \xrightarrow {S} \m V_k \xleftarrow {id_{\m V_k}} \m V_k,$$ 
together with a contravariant functorial characteristic class
$$c\ell(E) \in  CH^*(-)\otimes R\quad \text{or $\quad c\ell(E) \in  H^{2*}(-;R)\:\:$ for $k=\bC$}$$ 
of (isomorphism classes of) algebraic  vector bundles, which is 
\emph{multiplicative} in the sense that $$c\ell(E) = c\ell(E') \cup c\ell(E'')$$ for any short exact sequence
$0\to E'\to E\to E''\to 0$ of such vector bundles.
Here $CH^*(-)$ is the operational Chow cohomology group of
\cite{Fulton-book}. For a smooth algebraic manifold $V$, let 
$$\alp(V) := c\ell(TV) \cap [V]\in \mH_*(-)\otimes R,$$
with $[V]$ the fundamental class of $V$ for $\mH_*=CH_*$ the Chow group or 
$H_{2*}^{BM}$ the even degree Borel-Moore homology in the case of $k=\bC$.
Then the isomorphism invariant $\alp$ is \emph{additive and multiplicative}. 
By Theorem \ref{main},
there exists a unique natural transformation
$$\tau_{c\ell}: K^{prop}(\m V_k^{sm} \xrightarrow {S} \m V_k /-) \to \mH_*(-)\otimes R$$
such that for a smooth algebraic manifold $V$
$$\tau_{c\ell}([(V, V, id_{V})]) = c\ell(TV) \cap [V].$$
And $\tau_{c\ell}$ is also multiplicative (by Proposition \ref{mult}), i.e.:
$$\tau_{c\ell}([(V, X, h)] \times [(W, Y, k)]) = \tau_{c\ell}([(V, X, h)]) \boxtimes \tau_{c\ell}([(W, Y, k)]).$$
\end{cor}

\begin{rem} \label{rem-LM}
If we only consider projective morphisms and (pure dimensional) quasi-projective smooth varieties, then a similar result holds for
$\mH_*$ an \emph{oriented Borel-Moore weak homology theory} in the sense of \cite{Levine-Morel} and
$c\ell$ a multiplicative characteristic class as in \cite[\S 4.1.8]{Levine-Morel},
which 
is, for a line bundle $L$, given by a normalized power series in the first Chern class operator of $L$
with respect to $\mH_*$:
$$c\ell(L)=f(\tilde{c}_1(L)), \quad \text{with $f(t)\in 1+t\cdot \mH_*(pt)[[t]]$.}$$
Here the fundamental class of a quasi-projective smooth variety $V$ of pure dimension $d$ is defined as
$$[V]:=k^*1_{pt} \in \m H_d(V) \quad \text{for $k: V\to pt$ the constant smooth morphism.}$$
\end{rem}

Next we consider the subcategory $ \mC= \m V^{(l)pd}_k$ of (locally) pure-dimensional algebraic varieties over $k$. 
Let us first recall the universal property of the \emph{Nash blow-up}
$\nu: \widehat X \to X$ of a pure $d$-dimensional algebraic variety $X\in Obj(V^{pd}_k)$, with 
 $\widehat {TX}$ the \emph{tautological Nash tangent bundle} over $\widehat X$:
Let $\pi: \overline{X}\to X$ be a proper birational map with a surjection $\pi^*\Omega^1_X\to \overline{\Omega}$
to a locally free sheaf $\overline{\Omega}$ of rank $d$ on $\overline{X}$. Then the Nash blow-up
$\nu: \widehat X \to X$ is 
universal in the sense that 
$\pi: \overline{X}\to X$ factors through $\pi': \overline{X}\to \widehat X$,
with $\overline{\Omega}\simeq \pi'^*\widehat \Omega$, where the tautological Nash tangent bundle
 $\widehat {TX}$ over $\widehat X$ corresponds to the dual of $\widehat \Omega$.

\begin{defn} Let $c\ell$ be a functorial characteristic class of algebraic vector bundles
as in Corollary \ref{cor-alg}. For a pure $d$-dimensional algebraic variety $X\in Obj(V^{pd}_k)$, the \emph{$c\ell$-Mather homology class} 
$c\ell^{Ma}_*(X)\in \mH_*(X)\otimes R$ is defined to be 
$$c\ell^{Ma}_*(X):= \nu_*(c\ell(\widehat {TX}) \cap [\widehat X])=\pi_*(c\ell(\overline{TX}) \cap [\overline{X}]).$$
Here $\pi: \overline{X}\to X$ is any \emph{proper birational} map with a surjection $\pi^*\Omega^1_X\to \overline{\Omega}$
to a locally free sheaf $\overline{\Omega}$ of rank $d$ on $\overline{X}$, with the vector bundle $\overline{TX}$ corresponding to the dual of $\overline{\Omega}$. This definition is extended to a locally pure-dimensional variety $X\in Obj(V^{lpd}_k)$ by additivity over the connected components of $X$.
\end{defn}

Note that the second equality in the definition above follows from the projection formula by $\overline{TX}\simeq \pi'^*\widehat {TX}$
and $\pi'_*[\overline{X}]=[\widehat X]$, since $\pi': \overline{X}\to \widehat X$ is a proper birational map
(
see \cite[Example 4.2.9(b)]{Fulton-book} 
in the case of the Chern-Mather class $c^{Ma}$ corresponding to
$c\ell=c$ the Chern class).

\begin{cor}\label{mather} 
Consider the cospan of categories
$$\m V_k^{(l)pd} \xrightarrow {S} \m V_k \xleftarrow {id_{\m V_k}} \m V_k,$$ 
together with a contravariant functorial characteristic class
$$c\ell(E) \in  CH^*(-)\otimes R\quad \text{or $\quad c\ell(E) \in  H^{2*}(-;R)\:\:$ for $k=\bC$}$$ 
of (isomorphism classes of) algebraic  vector bundles, which is 
\emph{multiplicative}.
Here $CH^*(-)$ is the operational Chow cohomology group 
\cite[Definition 17.3]{Fulton-book}. For a (locally) pure-dimensional algebraic variety $V$, let 
$$\alp(V) := c\ell^{Ma}_*(V)\in \mH_*(V)\otimes R,$$
 for $\mH_*=CH_*$ the Chow group or 
$H_{2*}^{BM}$ the even degree Borel-Moore homology in the case of $k=\bC$.
\begin{enumerate}
\item The isomorphism invariant $\alp$ is \emph{additive and multiplicative}. 
By Theorem \ref{main},
there exists a unique natural transformation
$$\tau_{c\ell^{Ma}_*}: K^{prop}(\m V_k^{(l)pd} \xrightarrow {S} \m V_k /-) \to \mH_*(-)\otimes R$$
such that for a (locally) pure-dimensional algebraic variety $V$
$$\tau_{c\ell^{Ma}_*}([(V, V, id_{V})]) = c\ell^{Ma}_*(V).$$
And $\tau_{c\ell^{Ma}_*}$ is also multiplicative (by Proposition \ref{mult}), i.e.:
$$\tau_{c\ell^{Ma}_*}([(V, X, h)] \times [(W, Y, k)]) = \tau_{c\ell^{Ma}_*}([(V, X, h)]) \boxtimes \tau_{c\ell^{Ma}_*}([(W, Y, k)]).$$
\item When $c\ell = c$ is the Chern class, then the following diagram commutes for $k$ of characteristic zero:
$$\xymatrix{K^{prop} (\m V_k^{(l)pd} \xrightarrow {S} \m V_k/X) \ar[dr]_ {\tau_{c^{Ma}_*}}\ar[rr]^ {\m{E}u} && F(X) \ar[dl]^{c_*^{Mac}}\\& \mH_*(X)\:.}$$
Here the multiplicative natural transformation $\m{E}u: K^{prop}(\m V_k^{(l)pd} \xrightarrow {C} \m V_k /X)\to F(X)$ to the group  of constructible functions is defined by the isomorphism invariant $\alp(V):=Eu_V$, where $Eu_V$ is the local Euler obstruction of the (locally) pure-dimensional variety $V$.
This invariant is also additive and multiplicative.
\end{enumerate}
\end{cor}

\begin{proof} (1) That the isomorphism invariant $\alp(V) := c\ell^{Ma}_*(V)$ is additive follows from the fact,
that it commutes with restriction to open subsets (e.g. a connected component). For the multiplicativity we can then assume
that $V,V'$ are pure dimensional. Let $\pi: \overline{V}\to V$ be a proper birational map with a surjection $\pi^*\Omega^1_X\to \overline{\Omega}$
to a locally free sheaf $\overline{\Omega}$, and similarly for $V'$. Then $\pi\times \pi': \overline{V}\times \overline{V'}\to V\times V'$
is a proper birational map with a surjection
$$(\pi\times \pi')^*\Omega^1_{V\times V'}\simeq \pi^*\Omega^1_V\boxtimes \pi'^*\Omega^1_{V'}\to \overline{\Omega}\boxtimes  \overline{\Omega'}$$
so that
\begin{align*}
c\ell^{Ma}_*(V\times V')&= (\pi\times \pi')_*\left( c\ell (\overline{TV}\boxtimes  \overline{TV'})\cap [\overline{V}\times \overline{V'}] \right)\\
&= (\pi\times \pi')_*\left( (c\ell(\overline{TV})\cap [\overline{V}])\boxtimes ( c\ell(\overline{TV'})\cap [\overline{V'}]) \right)\\
&= \left( \pi_*(c\ell(\overline{TV})\cap [\overline{V}])\right) \boxtimes  \left( \pi'_*(c\ell(\overline{TV'})\cap [\overline{V'}] )\right)\\
&= c\ell^{Ma}_*(V) \boxtimes c\ell^{Ma}_*(V').
\end{align*}
(2) follows from the construction \cite{MacPherson} of the Chern class transformation 
$$c_*^{Mac}: F(X) \to H^{BM}_{2*}(X;\bZ)$$
for the case $k=\bC$.
MacPherson defined $c_*^{Mac}(Eu_V) = (i_V)_*c_*^{Ma}(V)$ for $i_V: V\to X$ the inclusion of a pure-dimensional subvariety,
with $Eu_V$ his famous local Euler obstruction. Namely, the ``constructible function" counterpart of the Chern--Mather homology class has to be the local Euler obstruction, which is one of his key observations. This is extended to locally pure-dimensional subvarieties by additivity over connected components,
e.g. the local Euler obstruction is then by definition the sum of the local Euler obstructions of all connected components.
The algebraic counterpart of the MacPherson Chern class transformation
$$c_*^{Mac}: F(X) \to CH_*(X)$$
for a base field $k$ of characteristic zero was constructed in \cite{Ke} (at least if $X$ is embeddable into a smooth variety.
The general case can be reduced to this using the method of Chow envelopes as in \cite[Chapter 18.3]{Fulton-book}).
Moreover, Kennedy also explained in \cite {Ke} that $Eu_V\in F(V)$ is a constructible function, where he used the algebraic definition
of the local Euler obstruction as in \cite[Example 4.2.9] {Fulton-book}(due to Gonzalez-Sprinberg and Verdier):
$$Eu_V(p)=\int_{\pi^{-1}(p)} c \left(\overline{TV}|_ {\pi^{-1}(p)} \right) \cap s \left(\pi^{-1}(p),\overline{V} \right),$$
with $\pi: \overline{V}\to V$ and $\overline{TV}$ are as in the definition of the $c\ell$-Mather homology classes. 
Here $s\bigl(\pi^{-1}(p),\overline{V} \bigr)$ is the Segre class of the fiber $\pi^{-1}(p)$ in $\overline{V}$ in the sense of
\cite[Chapter 4.2]{Fulton-book}. Then the multiplicativity of $Eu_V$ follows as in (1) using \cite[Example 4.2.5]{Fulton-book}:
$$s \Bigl(\pi^{-1}(p)\times \pi'^{-1}(p'),\overline{V}\times \overline{V'} \Bigr)=s \bigl(\pi^{-1}(p),\overline{V} \bigr)\boxtimes s \bigl(\pi'^{-1}(p'),\overline{V'} \bigr).$$
Here the bilinear (functorial) cross product 
$$\boxtimes: F(X)\times F(Y)\to F(X\times Y)$$
is just defined by $\beta\boxtimes \beta' ((p,p')):=\beta(p)\cdot \beta'(p')$.
\end{proof}

\begin{rem} Assume that the base field $k$ is of characteristic zero.
\begin{enumerate}
\item Using resolution of singularities one can show that for a given algebraic variety $X$ there are finitely many irreducible subvarieties $V$'s and integers $a_V$'s such that
$$\jeden_X = \sum_{V \subset X} a_V Eu_V \quad \text{, thus} \quad c_*^{Mac}(\jeden_X) = \sum_{V \subset X} a_V c_*^{Ma}(V).$$  Whether $X$ is singular or not, $c_*^{Mac}(X):=c_*^{Mac}(\jeden_X)$ is called MacPherson's Chern class or Chern--Schwarz--MacPherson class of $X$ (see \cite{BrSc, Ke, MacPherson, Schw1, Schw2}). 
For $X$ complete, it follows from the naturality of the transformation $c_*^{Mac}$ with respect to the proper constant map
$X\to pt$, that the degree of the $0$-dimensional component of $c_0^{Mac}(X)$ is equal to the Euler--Poincar\'e characteristic:
$$\int_X c_0^{Mac}(X) = \chi(X).$$
 \item Similarly, the degree of the $0$-dimensional component of the Chern--Mather class $c_*^{Ma}(X)$ for $X$ pure-dimensional and complete
is the Euler--Poincar\'e characteristic of $X$ weighted by the local Euler obstruction $Eu_X$:
 $$\int_X c_0^{Ma}(X) = \chi(X;Eu_X).$$
For $X$ a connected \emph{complex affine} algebraic variety of pure dimension,  
the \emph{global Euler obstruction $Eu(X)$} introduced and studied in \cite{STV} is a suitable ``localization'' of the $0$-dimensional component of the Chern--Mather class $c_*^{Ma}(X)$:
 $$H_0(X;\bZ)\simeq \bZ \ni  Eu(X)\mapsto  c_0^{Ma}(X) \in H^{BM}_0(X;\bZ) $$
under the natural map $H_0(X;\bZ)\to H^{BM}_0(X;\bZ) $.
 \item The above ``$c\ell$-Mather class" transformation $\tau_{c\ell^{Ma}_*}: K^{prop}(\m V_k^{(l)pd} \xrightarrow {S} \m V_k /-) \to \mH_*(-)\otimes R$
 could be considered as  a very na\"\i ve theory of characteristic classes of possibly singular  algebraic varieties (in any characteristic).
 \end{enumerate}
\end{rem}

So far we dealt with the covariance and multiplicativity of the functor $K_{(\alp)}(\mC_s \xrightarrow {S} \mB /T(-))$. Next we discuss the contravariance
with respect to (suitable) ``smooth morphisms'', where we start with the algebraic geometric context with $\mB=\m V_k$ the category of algebraic varieties over $k$ and $S: \mC\to \m V_k$ the inclusion functor of the (full) subcategories $\mC= \m V_k^{sm}$ resp. $\m V_k^{lci}$ of \emph{smooth} resp. \emph{local complete intersection} varieties.

Here $X$ is called a local complete intersection, if it has a regular closed embedding $i: X\to M$ into a smooth variety $M$
(i.e. the constant morphism $X\to pt$ is a local complete intersection morphism in the sense of \cite[Chapter 6.6]{Fulton-book}).
Then $X$ has an intrinsic virtual tangent bundle 
$$TX:=i^*TM - N_XM \in K^0(X),$$
with $N_XM$ the normal bundle of the regular embedding $i: X\to M$, i.e. $TX\in K^0(X)$ doesn't depend on the choice of this embedding
(compare \cite[Appendix B.7.6]{Fulton-book}). Of course any smooth variety $M$ is local complete intersection, with $TM$ (the class of) the usual
tangent bundle $TM$ (just choose $i=id_M: M\to M$). Note that the (virtual) tangent bundle commutes with restriction to open subsets
(e.g. connected components). Similarly, if $f: X\to Y$ is a smooth morphism with $Y$ smooth (resp. a local complete intersection),
then also $X$ is smooth (resp. a local complete intersection), and in the smooth context we have a short exact sequence of vector bundles
\begin{equation}\label{eq-s1}
 0\to T_f \to TX \to f^*TY\to 0.
\end{equation}
with $T_f$ the bundle of tangents to the fiber of the smooth morphism $f$. In particular
\begin{equation}\label{eq-s2}
TX = f^*TY + T_f \in K^0(X),
\end{equation}
and this equality in the Grothendieck group $K^0(-)$ of algebraic vector bundles even holds for a smooth morphism $f: X\to Y$ between local complete
intersections with $TX$ resp. $TY$ the corresponding virtual tangent bundle (compare e.g. with \cite[Proposition 7.1] {Fulton-Lang}as well as 
\cite[Appendix B.7]{Fulton-book}). Finally, the class of smooth morphisms is stable under base-change, with $T_{f'}\simeq h'^*T_f$ for a fiber square as in the following

\begin{lem} \label{lemma} The functor $K^{prop}(\mC \xrightarrow {S} \m V_k /-)$, with $\mC= \m V_k^{sm}$ (resp. $\m V_k^{lci}$) 
the subcategory of smooth (resp. local complete intersection) varieties,
becomes a contravariant functor for smooth morphisms on the category $\m V_k$, where for a smooth morphism $f:X \to Y$ the pullback homomorphism 
$$f^*: K^{prop}(\mC \xrightarrow {S} \m V_k /Y)\to K^{prop}(\mC \xrightarrow {S} \m V_k /X)$$
is defined by 
$$f^*([(V, Y, h)]) := [V',X,h'],$$
using the fiber square
$$\CD
V'@> {f'} >> V\\
@V h' VV @VV h V\\
X@> {f} >> Y. \endCD
$$
\end{lem}

Of course here we also use the fact that taking such fiber squares commutes with disjoint unions in $V$ resp. $V'$.

\begin{thm}[Verdier-type Riemann--Roch] \label{Verdier-RR}
Consider the cospan of categories
$$\mC \xrightarrow {S} \m V_k \xleftarrow {id_{\m V_k}} \m V_k,$$ 
with $\mC= \m V_k^{sm}$ (resp. $\m V_k^{lci}$) 
the subcategory of \emph{smooth} (resp. \emph{local complete intersection}) varieties.
Let $c\ell$ be a  contravariant functorial characteristic class
$$c\ell(E) \in  CH^*(-)\otimes R\quad \text{or $\quad c\ell(E) \in  H^{2*}(-;R)\:\:$ for $k=\bC$}$$ 
of (isomorphism classes of) algebraic  vector bundles, which is 
multiplicative (and normalized in case $\mC= \m V_k^{lci}$, so that it can also be defined on virtual vector bundles in $K^0(-)$).
For a smooth (resp. local complete intersection) variety $V$, let 
$$\alp(V) := c\ell(TV)\cap [V]\in \mH_*(V)\otimes R$$
be the corresponding (virtual) characteristic homology class, with
$\mH_*=CH_*$ the Chow group or 
$H_{2*}^{BM}$ the even degree Borel-Moore homology in case $k=\bC$.

The isomorphism invariant $\alp$ is \emph{additive}, so that  
by Theorem \ref{main}
there exists a unique natural transformation
$$\tau_{c\ell}: K^{prop}(\mC \xrightarrow {C} \m V_k /-) \to \mH_*(-)\otimes R,$$
such that for a smooth (resp. local complete intersection) algebraic variety $V$
$$\tau_{c\ell}([(V, V, id_{V})]) =  c\ell(TV)\cap [V].$$
Then this natural transformation $\tau_{c\ell}$ satisfies the following Verdier-type Riemann--Roch formula: For a smooth morphism $f:X \to Y$ the following diagram commutes:
$$\CD
K^{prop}(\mC \xrightarrow {S} \m{V}_k /Y)  @> {\tau_{c\ell}} >> \mH_*(Y)\otimes R\\
@V f^* VV @VV c\ell(T_f) \cap f^* V\\
K^{prop}(\mC \xrightarrow {S} \m{V}_k /X) @>>  {\tau_{c\ell}}  > \mH_*(X)\otimes R. \endCD
$$
\end{thm}

\begin{proof} On one hand we have:
\begin{align*}
\tau_{c\ell}(f^*([(V, Y, h)])) & = \tau_{c\ell}([(V', X, h')]) \\
& = h'_*(c\ell(TV') \cap [V']). 
\end{align*}
On the other hand we have
\begin{align*}
& c\ell(T_f) \cap  f^*(\tau_{c\ell}([(V, Y, h)]) \\
& = c\ell(T_f) \cap f^* (h_* (c\ell (TV) \cap [V]))
\end{align*}

For a  fiber square
$$\CD
V'  @> f'>> V \\
@V h' VV @VV h V\\
X@>>  f  > Y \endCD$$
with $h:V \to Y$ proper and $f:X \to Y$ smooth, 
we have the base change identity (see \cite[Proposition 1.7]{Fulton-book}):
$$f^*h_* = h'_*{f'}^*: \mH_*(V)\otimes R \to \mH_*(X)\otimes R \:.$$
Hence the above equality continues as follows:
\begin{align*}
& = c\ell(T_f) \cap h'_*{f'}^* (c\ell (TV) \cap [V]))\\
& = h'_* ({h'}^*c\ell(T_f) \cap {f'}^* (c\ell (TV) \cap [V])) \, \text {(by the projection formula)}\\
& = h'_* (c\ell(T_{f'}) \cap {f'}^* (c\ell (TV) \cap [V])) \, \text {(by $T_{f'}\simeq h'^*T_f$)}\\
& = h'_* \left (c\ell(T_{f'}) \cap (c\ell ({f'}^*TV) \cap {f'}^*[V]) \right ) \, \text {(by functoriality of $c\ell(-)\cap$)}\\
& = h'_* \left( (c\ell(T_{f'}) \cup c\ell ({f'}^*TV)) \cap [V']\right)\\
&= h'_* \left (c\ell(TV') \cap [V'] \right ) 
\end{align*}
by the multiplicativity of $c\ell$ and (\ref{eq-s1}) resp. (\ref{eq-s2}). Of course we also used the relation
$[V']={f'}^*[V]$ for the fundamental classes.
Therefore we get that 
$$\tau_{c\ell}(f^*([(V, Y, h)]))  = c\ell(T_f) \cap f^*(\tau_{c\ell}([(V, Y, h)]),$$
and the above diagram of the theorem commutes.
\end{proof}

\begin{rem} \label{rem-LM2}
If we 
consider only projective morphisms and (pure dimensional) quasi-projective local complete intersection (resp. smooth)
varieties, then a similar Verdier-type Riemann--Roch formula holds for
$\mH_*$ an \emph{oriented Borel-Moore (weak) homology theory} in the sense of \cite{Levine-Morel} and
$c\ell$ a multiplicative characteristic class as in \cite[\S 7.4.1 resp. \S 4.1.8]{Levine-Morel},
which 
is, for a line bundle $L$, given by a normalized power series in the first Chern class operator of $L$
with respect to $\mH_*$:
$$c\ell(L)=f(\tilde{c}_1(L)), \quad \text{with $f(t)\in 1+t\cdot \mH_*(pt)[[t]]$.}$$
Here the fundamental class of a quasi-projective local complete intersection (resp. smooth) variety $V$ of pure dimension $d$ is defined as
$$[V]:=k^*1_{pt} \in \m H_d(V) \quad \text{for $k: V\to pt$ the constant local complete intersection (resp. smooth) morphism}$$
so that $[V']={f'}^*[V]$ for a smooth morphism $f': V'\to V$ by functoriality of $f'^*$.
\end{rem}

\begin{rem} Assume that the base field $k$ is of characteristic zero.
\begin{enumerate}
\item The interesting thing about the motivic Hirzebruch class or MacPherson Chern class transformation
$${T_y}_*: K^{prop}(\m V^{sm}_k\xrightarrow {S} \m V_k /-) \to \m H_*(-)\otimes \bQ[y]
\quad \text{or} \quad
c^{Mac}_*: K^{prop}(\m V^{sm}_k\xrightarrow {S} \m V_k /-) \to \m H_*(-)$$
is for example, that in the above discussions, the na\" \i ve relative Grothendieck group $K^{prop}(\m V^{sm}_k\xrightarrow {S} \m V_k /X)$ can be replaced by the much smaller and more interesting relative Grothendieck group $K_0(\m V_k/X)$, by imposing one more \emph{ additivity relation}, as recalled in the introduction:
$$[V \xrightarrow h X] = [W \xrightarrow {h|_W} X] + [V \setminus W \xrightarrow {h|_{V \setminus W}} X],$$
with $W \subset V$ is a closed subvariety of $V$. So these two transformations factorize over the tautological surjective transformation
$$ K^{prop}(\m V^{sm}_k\xrightarrow {S} \m V_k /-) \to K_0(\m V_k/-).$$
In particular the multiplicativity result of Corollary \ref{cor-alg} and the Verdier-type Riemann--Roch formula of Theorem \ref{Verdier-RR}
are true for the motivic Hirzebruch class and the MacPherson Chern class transformation (compare \cite{BSY}).
\item The corresponding Hirzebruch class $T_{y*}(X)=T_{y*}(id_X)$ and MacPherson Chern class $c_*^{Mac}(X)=c_*^{Mac}(id_X)$
of an algebraic variety $X$ is not only invariant under an isomorphism, but also under a proper \emph{(geometric) bijection}:
$$f_*\left(T_{y*}(X)\right)=T_{y*}(Y) \quad \text{and} \quad f_*\left(c_*^{Mac}(X)\right)=c_*^{Mac}(Y)$$
for a proper morphism $f: X\to Y$ such that the induced map $f: X(\bar{k}) \to Y(\bar{k})$ is a bijection of sets (with $\bar{k}$ an algebraic closure of $k$),
since then $f_*([id_X])=[id_Y]$ by Noetherian induction using additivity and generic smoothness of $f$.
Of course for MacPherson's Chern class $c_*^{Mac}(X)=c_*^{Mac}(\jeden_X)$
this also follows from $f_*\jeden_X=\jeden_Y\in F(Y)$ for a proper (geometric) bijection $f: X\to Y$.
Similarly,  Baum--Fulton--MacPherson's Todd class $td^{BFM}_*(X)=td^{BFM}_*([\m O_X])$ is invariant under a such a proper (geometric) bijection
(for $k$ of any characteristic),
since $f_*[\m O_X]=[\m O_Y]\in G_0(Y)$. Finally, 
Goresky--MacPherson's or Cappell--Shaneson's homology $L$-class $L_*(X)=L_*^{CS}([\m {IC}_X])$ of a (locally) pure-dimensional compact complex algebraic variety $X$ is not only invariant under a  proper bijection, but more generally under taking the normalization of $X$. 
\end{enumerate} \end{rem}

Let us finish this section with a counterpart in the differential-topological context, with $\m B=\m C^{\infty}$ the category 
of $\m C^{\infty}$-manifolds and $S: \mC\to \m C^{\infty}$ the forget functor from the 
the category $\mC = \mC_{(o)}^{\infty}$ or $\mC^{\infty}_{\bC}$ of all differentiable (oriented) or stable complex $\m C^{\infty}$-manifolds.

\begin{lem} \label{lemma2} The functor $K^{prop}(\mC \xrightarrow {S} \m C^{\infty}/-)$ for $\mC = \mC_{(o)}^{\infty}$ or $\mC^{\infty}_{\bC}$
becomes a contravariant functor for a ((complex) oriented) submersion  $f:X \to Y$ of $\m C^{\infty}$-manifolds.
Here the pullback homomorphism 
$$f^*: K^{prop}(\mC \xrightarrow {S} \m C^{\infty} /Y)\to K^{prop}(\mC \xrightarrow {C} \m C^{\infty}/X)$$
is defined by 
$$f^*([(V, Y, h)]) := [V',X,h'],$$
using the fiber square
$$\CD
V'@> {f'} >> V\\
@V h' VV @VV h V\\
X@> {f} >> Y, \endCD
$$
with $TV'$ oriented (or given a stable complex structure) by the short exact sequence 
$$ 0\to T_{f'} \to TV' \to f'^*TV\to 0,$$
 if $TV$ and the bundle of tangents to the fibers $T_f$ (and therefore also $T_{f'}\simeq h'^*T_f$)
are oriented (or stable complex).
\end{lem}

Then the proof of the following result is identical to that in the algebraic geometric context:
\begin{thm}[Verdier-type Riemann--Roch for smooth manifolds] \label{Verdier-RR-smooth}
Consider the cospan of categories
$$\mC \xrightarrow {S} \m C^{\infty} \xleftarrow {id_{\m C^{\infty}}} \m C^{\infty},$$ 
with $\mC = \mC_{(o)}^{\infty}$ or $\mC^{\infty}_{\bC}$ the category of all differentiable (oriented) or stable complex $\m C^{\infty}$-manifolds.
Let $c\ell$ be a  contravariant functorial multiplicative and normalized characteristic class
$$c\ell(E) \in  H^*(-;\bZ_2)\quad \text{or} \quad c\ell(E) \in  H^{2*}(-;R)$$ 
of (isomorphism classes of) real (oriented) or complex vector bundles.
For a smooth (oriented or stable complex) manifold $V$, let 
$$\alp(V) := c\ell(TV)\cap [V]\in H^{BM}_*(V;R).$$
The invariant $\alp$ is \emph{additive}, so that  
by Theorem \ref{main}
there exists a unique natural transformation
$$\tau_{c\ell}: K^{prop}_{\alp}(\mC \xrightarrow {S}  \m C^{\infty}/-) \to  H^{BM}_*(-;R),$$
such that for a smooth (oriented or stable complex) manifold $V$
$$\tau_{c\ell}([(V, V, id_{V})]) =  c\ell(TV)\cap [V].$$
Then this natural transformation $\tau_{c\ell}$ satisfies the following Verdier-type Riemann--Roch formula: For a ((complex) oriented) submersion 
 $f:X \to Y$  of $\m C^{\infty}$-manifolds, the following diagram commutes:
$$\CD
K^{prop}_{\alp}(\mC \xrightarrow {S} \m C^{\infty}/Y)  @> {\tau_{c\ell}} >> H^{BM}_*(Y;R)\\
@V f^* VV @VV c\ell(T_f) \cap f^* V\\
K^{prop}_{\alp}(\mC \xrightarrow {S} \m C^{\infty}/X) @>>  {\tau_{c\ell}}  > H^{BM}_*(X;R). \endCD
$$
\end{thm}

\section{Examples}

In this last section we discuss some results, questions and problems related to some very specific examples,
which also fit with our natural transformations associated to some additive (homology) invariants.

\subsection{The case of the fundamental class}
Let us first consider taking the fundamental class $[-]\in H^{BM}_*(-;R)$ for $R=\bZ_2$ (or $R=\bZ$)  on the category $\mC^{\infty}_{(o)}$ of (oriented) $C^{\infty}$-manifolds. The fundamental class $\alp:=[-]$ is certainly an additive and multiplicative homology class and by Corollary \ref{smooth-gen} we have a unique natural transformation
$$\tau_{[-]}: K^{prop}_{\alp}(\mC \xrightarrow {S} \mTOP_{lc} /-) \to H^{BM}_*(-;R)$$
such that $\tau_{[-]}([(V, V, id_{V})]) = [V]$
for a smooth (oriented) manifold $V$. Here an $\alp$-isomorphism $\phi: V\to V'$ is just an (orientation preserving) diffeomeorphism
by Remark \ref{or-preserving}.

Of course such an additive (and multiplicative) fundamental homology class $\alp:=[-]$ is
also available on bigger categories of spaces, like the categories $\mC^{0}_{(o)}$ of (oriented) topological manifolds, or
 $\m C^{\text{pseudo}}_{(o)}$  the category of (oriented) triangulated or topological pseudo-manifolds, whose morphisms are by definition just
continuous maps of the underlying topological spaces (so that an isomorphism is nothing but a homeomorphism).
If we restrict ourselfes to compact spaces (as classically often done), then these fundamental classes live in usual homology
$H_*(-;R)$ for $R=\bZ_2$ (or  $R=\bZ$ in the oriented context).
Moreover such an additive and multiplicative fundamental class $\alp=[-]\in H_*(-;R)$ is also available for the category
$\m C^{\text{Poincar\'e}}_{(o)}$  of finite  (oriented) Poincar\'e complexes, i.e., topological spaces $V$ (like homology-manifolds) which satisfy Poincar\'e duality
$$\cap [V]: H^*(V;R)\stackrel{\sim}{\to} H_*(V;R)$$
for a suitable fundamental class $[V]\in H_*(V;R)$ . So we can turn these categories $\m C$ as before into symmetric monoidal categories with respect to disjoint unions $\sqcup$
(and also products $\times$) such that the forgetful functor $\frak f: \m C\to \mTOP_{lc}$ is strictly monoidal.
Moreover, an $\alp$-isomorphism $\phi: V\to V'$ in this context is just an (orientation preserving) homeomorphism.\\

Then the classical \emph{Steenrod's realization problem} can be reinterpreted as the problem of asking for the surjectivity of the group homomorphism $\tau_{[-]}: K_{\alp}(\mC \xrightarrow {\frak f} \mTOP /X) \to H_*(X;R)$ for a topological space $X$.
Here the following results are well known (see \cite{Rudyak2, Sullivan2}).
\begin{thm}\label{rudyak}
\begin{enumerate}
\item (\cite{Thom} and \cite[Chapter IV, Theorem 7.33]{Rudyak})
$$\tau_{[-]}: K(\mC^{\infty}_c \xrightarrow {\frak f} \mTOP /X) \to  H_*(X;\bZ_2)$$
 is surjective, i.e. every $\bZ_2$-homology class can be realized by a compact smooth manifold.
\item (\cite{Thom} and \cite[Chapter IV, Theorem 7.37]{Rudyak}) The following composed map is surjective:
$$\tau_{[-]}: K_{\alp}(\mC^{\infty}_{co} \xrightarrow {\frak f} \mTOP /X) \to H_*(X;\bZ) \stackrel{proj.}{\to}
\bigoplus _{0 \leq i \leq 6} H_i(X;\bZ),$$
i.e. every $\bZ$-homology class in degree $i\leq 6$ can be realized by a compact oriented smooth manifold.
\item (\cite{Levitt}) The following composed map is surjective:
$$\tau_{[-]}: K_{\alp}(\m C^{\text{Poincar\'e}}_{co} \xrightarrow {\frak f} \mTOP /X) \to H_*(X;\bZ) \stackrel{proj.}{\to}
\bigoplus _{i \not =3 } H_i(X;\bZ),$$
i.e. every $\bZ$-homology class in degree $i\neq 3$ can be realized by a finite oriented Poincar\'e complex.
\item (\cite{Sullivan} and \cite[Chapter VIII, Example 1.25(a)]{Rudyak})  
$$\tau_{[-]}: K_{\alp}(\m C^{\text{pseudo}}_{co} \xrightarrow {\frak f} \mTOP /X) \to H_*(X;\bZ)$$
 is surjective, i.e. every $\bZ$-homology class can be realized by a compact oriented pseudo-manifold.
\end{enumerate}
\end{thm}

\subsection{The case of the Stiefel--Whitney class}
Let $V$ be a differentiable manifold and let $c\ell_*(V) \in H^{BM}_*(V, \bZ_2)$ be the Poincar\'e dual $c\ell(TV) \cap [V]$ of a 
multiplicative and normalized functorial characteristic class $c\ell(E)\in H^*(-, \bZ_2)$ of (isomorphism classes of) real vector bundles.
$\alp(V):=c\ell_*(V)$ is clearly an additive (and multiplicative) homology class,
and we have by Corollary \ref{smooth-gen}
 a unique natural transformation
$$\tau_{c\ell_*}: K^{prop}(\mC^{\infty} \xrightarrow {S} \mTOP_{lc} /-) \to H^{BM}_*(-;\bZ_2)$$
such that $\tau_{c\ell_*}([(V, V, id_{V})]) = c\ell_*(V)$
for a smooth  manifold $V$.
In particular the Stiefel--Whitney class $c\ell=w$ is a typical one. \\

 If we restrict ourselves to the category $\m V_{\bR}$ of real algebraic varieties and let $\m V^{sm}_{\bR}$ be its full subcategory of smooth real algebraic varieties, then we have a more geometric ``realization" of the natural transformation $w_*$ on the category $\m V_{\bR}$
through \emph{$\bZ_2$-valued semi-algebraic constructible functions}:
$$\xymatrix{K^{prop} (\m V^{sm}_{\bR} \xrightarrow {\iota} \mathcal {V_{\bR}} /X) \ar[dr]_ {w_*}\ar[rr]^ {const} && F(X;\bZ_2) \ar[dl]^{w_*}\\& 
H^{BM}_*(X, \bZ_2)\:.}$$
Here the multiplicative natural transformation $const$ is defined by the isomorphism invariant $\alp(V):=1_V$, which is additive and multiplicative.
The Stiefel-Whitney class transformation 
$w_*: F(X;\bZ_2)\to H^{BM}_*(X, \bZ_2)$ was first constructed by Sullivan \cite{Sullivan} in the pl-context.
For a new approach in the semi-algebraic (or subanalytic) context through the theory of ``conormal or characteristic cycles'' 
see \cite{FMc}.

\begin{rem}\label{Thom-Whitney} For a compact topological manifold $V$, Thom constructed a Whitney class using a relation with Steenrod squares \cite{Thom-7} (see \cite {MacPherson2}). Let us denote  this  Thom-Whitney class in homology by $w_*^{Th}(V) \in H_*(V;\bZ_2)$,
which for a compact smooth manifold $V$ agrees with the Stiefel-Whitney class $w_*(V)$ above.
Then also $\alp(V):= w_*^{Th}(V)$ is an additive (and multiplicative) invariant, so that we have a natural transformation
$$\tau_{w_*}^{Th}: K(\mC^{0}_c  \xrightarrow {\frak f} \mTOP /-) \to H_*(-;\bZ_2)$$
defined by
$$\tau_{w_*}^{Th}([(V, X, h)]) = h_*w_*^{Th}(V).$$
If we consider the above Whitney class transformation 
$$w_*: K(\mC^{\infty}_c \xrightarrow {\frak f} \mTOP/-) \to H_*(-;\bZ_2)$$
for compact smooth $C^{\infty}$-manifolds, then this is surjective by Theorem \ref{rudyak}(i), since $w_*(V)=[V]+ \dots$ 
equals the fundamental class plus lower order terms. So a natural problem is to find
for a given compact topological manifold $X$ a class $\beta \in K(\mC^{\infty}_c\xrightarrow {\frak f} \mTOP/X)$ such that 
$$w_*(\beta) = w_*^{Th}(X).$$ 
\end{rem}

\subsection{The case of the Pontryagin and L-class}
Let $V$ be an oriented differentiable manifold and let $c\ell_*(V) \in H^{BM}_{2*}(V,\bQ)$ be the Poincar\'e dual $c\ell(TV) \cap [V]$ of a 
multiplicative and normalized functorial characteristic class $c\ell(E)\in H^{2*}(-, \bQ)$ of (isomorphism classes of) oriented real vector bundles.
$\alp(V):=c\ell_*(V)$ is clearly an additive (and multiplicative) homology class,
and by Corollary \ref{smooth-gen} we have 
 a unique natural transformation
$$\tau_{c\ell_*}: K^{prop}_{\alp}(\mC^{\infty}_o \xrightarrow {S} \mTOP_{lc} /-) \to H^{BM}_{2*}(-;\bQ)$$
such that $\tau_{c\ell_*}([(V, V, id_{V})]) = c\ell_*(V)$
for a smooth oriented manifold $V$.
In particular the Pontryagin class $c\ell=p$ and the Hirzebruch-Thom $L$-class $c\ell=L$ are typical ones. \\

If we  restrict ourselves to the category $\m  V_{\bC,c}$ of compact complex algebraic varieties
and let $\m V^{sm}_{\bC,c}$ be its full subcategory of compact smooth complex algebraic varieties, then we have a more geometric ``realization" of the natural transformation $L_*$ on the category $\m V_{\bC,c}$
 through 
\emph{Cappell--Shaneson--Youssin's cobordism groups $\Omega_*(X)$}(see \cite{CS}, \cite{Youssin}) and the motivic relative Grothendieck
group as mentioned in the introduction (compare with \cite{BSY} for more details):
$$\begin{CD}
K(\m V^{sm}_{\bC,c} \xrightarrow {\iota} \m V_{\bC,c}/X) @> taut >> K_0(\m V_{\bC}/X) \\
@V L_* VV @VV sd V \\
H_{2*}(X, \bQ) @<< L^{CS}_* < \Omega(X) .
\end{CD}$$
Here the (composed) multiplicative natural transformation 
$$sd: K(\m V^{sm}_{\bC,c} \xrightarrow {\iota} \m V_{\bC,c}/X) \to \Omega(X)$$ 
is defined by the isomorphism invariant $\alp(V):=[\bQ_V[dim\;V]]\in \Omega(V)$, which is additive and multiplicative.

\begin{rem} For a compact triangulated $\bQ$-homology manifold $V$, Thom constructed in \cite{Thom-8}
an $L$-class in $H^{2*}(V;\bQ) \simeq H_{2*}(V;\bQ)$
using a relation with the signature (see \cite {MacPherson2, Sullivan2}), so that for a compact oriented smooth manifold
it agrees with the usual $L$-class $L(TV)$ of the tangent bundle by the famous \emph{Hirzebruch signature theorem}.
Through the development of \emph{Intersection (co)homology}, this approach was further extended (by e.g. Goresky-MacPherson \cite{GM} and
Cappell-Shaneson \cite{CS})
to more general singular spaces
like compact (locally) pure-dimensional complex algebraic varieties. But then one has to view this class as a
homology $L$-class in  $H_{2*}(V;\bQ)$. Nevertheless, in Thom's original approach for a compact triangulated $\bQ$-homology manifold $V$ only,
one can use Poincar\'e duality to view this as a cohomology class in $H^{2*}(V;\bQ)$. Using the cup-product structure on cohomology,
he was then able to define also a Pontryagin class $p(V)\in H^{2*}(V;\bQ)$ so that for a compact oriented smooth manifold
it agrees with the usual Pontryagin class $p(TV)$ of the tangent bundle.

Let us call the corresponding homology classes the Thom--Pontryagin and Thom $L$-class, denoted by 
$$p_*^{Th}(X), L^{Th}_*(X) \in H_{2*}(X;\bQ).$$
If we consider the above Pontryagin- or $L$-class transformation for $c\ell=p$ or $L$ in the context of compact oriented smooth manifolds
$$\tau_{c\ell_*}: K_{\alp}(\mC^{\infty}_{co} \xrightarrow {S} \mTOP /-) \to H_{2*}(-;\bQ),$$
then for a given compact triangulated $\bQ$-homology manifold $X$ it is a very interesting  problem to find a class 
$\beta \in K(\mC^{\infty}_{co} \xrightarrow {S} \mTOP /X)$ such that 
$$p_*(\beta) = p_*^{Th}(X) \quad \text{or} \quad  L_*(\beta)=L^{Th}_*(X).$$ 
Note that the rationalized group homomorphism $\tau_{c\ell_*}\otimes \bQ$ is surjective by \cite{Thom} (compare \cite[Chapter IV, Theorem 7.36]{Rudyak}),
since $c\ell_*(V)=[V]+ \dots$ 
equals the fundamental class plus lower order terms.
\end{rem}

\subsection {The case of the Chern class}
Let $V$ be a stable complex differentiable manifold and let $c\ell_*(V) \in H^{BM}_{2*}(V,R)$ be the Poincar\'e dual $c\ell(TV) \cap [V]$ of a 
multiplicative and normalized functorial characteristic class $c\ell(E)\in H^{2*}(-, R)$ of (isomorphism classes of) complex vector bundles.
$\alp(V):=c\ell_*(V)$ is clearly an additive (and multiplicative) homology class,
and we have by Corollary \ref{smooth-gen}
 a unique natural transformation
$$\tau_{c\ell_*}: K^{prop}(\mC^{\infty}_{\bC} \xrightarrow {S} \mTOP_{lc} /-) \to H^{BM}_{2*}(-;R)$$
such that $\tau_{c\ell_*}([(V, V, id_{V})]) = c\ell_*(V)$
for a smooth stable complex manifold $V$.
In particular the Chern class $c\ell=c$ is a typical one. \\

If we  restrict ourselves to the category $\m  V_{\bC}$ of complex algebraic varieties
and let $\m V^{sm}_{\bC}$ be its full subcategory of  smooth complex algebraic varieties, then we have a more geometric ``realization" of the natural transformation $c_*$ on the category $\m V_{\bC}$
 through constructible functions
 and the motivic relative Grothendieck
group as mentioned in the introduction (
see \cite{BSY} for more details):
$$\begin{CD}
K^{prop}(\m V^{sm}_{\bC} \xrightarrow {\iota} \m V_{\bC}/X) @> taut >> K_0(\m V_{\bC}/X) \\
@V c_* VV @VV const V \\
H^{BM}_{2*}(X, \bZ) @<< c^{Mac}_* < F(X) .
\end{CD}$$
Here the (composed) multiplicative natural transformation 
$$const: K^{prop}(\m V^{sm}_{\bC} \xrightarrow {\iota} \m V_{\bC}/X) \to F(X)$$ 
is defined by the isomorphism invariant $\alp(V):=\jeden_V\in F(V)$, which is additive and multiplicative.








\subsection {The case of Chern classes of other types}

Let $\m V^{emb}_{\bC}$ be the subcategory of (complex) algebraic varieties embeddable into smooth varieties and let $c_*^{FJ}(X)$ resp., $c_*^{FJ}(X)$ be Fulton--Johnson's Chern class resp., Fulton's canonical class defined for such an embeddable (complex) algebraic variety:
 $c_*^{FJ}(X)$ (\cite[Example 4.2.6 (c)]{Fulton-book}) is defined by
$$c_*^{FJ}(X):= c(TM|_X) \cap s(\m N_XM),$$
where $TM$ is the tangent bundle of $M$ and $s(\m N_XM)$ is the Segre class of the conormal sheaf $\m N_XM$ of $X$ in $M$ \cite[\S 4.2]{Fulton-book}. Fulton's canonical class $c_*^F(X)$ (\cite[Example 4.2.6 (a)]{Fulton-book}) is defined by
$$c_*^F(X) := c(TM|_X) \cap s(X,M),$$
where $s(X,M)$ is the relative Segre class \cite[\S 4.2]{Fulton-book}.
For a local complete intersection variety $X$ we also have a normal bundle $N_XM$ in $M$, from which we can define the virtual tangent bundle $T_X$ of $X$ 
(as in \S 3) by
$$T_X := TM|_X - N_XM \in K^0(X).$$
As shown in \cite[Example 4.2.6]{Fulton-book}, for a local complete intersection variety $X$ in a non-singular variety $M$, these two Chern classes are both equal to the virtual Chern class
$$c_*^{FJ}(X) = c_*^F(X) = c(T_X) \cap [X].$$

Moreover, both isomorphism invariants $\alp(V):=c_*^{FJ}(V)$ and $c_*^{F}(V)$ are additive for $V\in ob(\m V^{emb}_{\bC})$, so
that there
exists  unique natural transformations on the category $\m V_{\bC}$:
$$\tau_{c_*^F}, \tau_{c_*^{FJ}}  : K^{prop}(\m V^{emb}_{\bC} \xrightarrow {\iota} \m V_{\bC}/-) \to H^{BM}_{2*}(-; \bZ),$$
such that 
$$\tau_{c_*^F}([V \xrightarrow{\op{id}_V}  V]) = c_*^F(V) \quad \text{resp.} \quad \tau_{c_*^{FJ}}  ([V \xrightarrow{\op{id}_V}  V]) = c_*^{FJ}(V)$$  
for $V \in ob(\m V^{emb}_{\bC})$.
Finally, using Chow groups $CH_*(X)$ as a corresponding homology theory, all of this remains true over any base field $k$ (instead of working over
$\bC$ with Borel-Moore homology), with the invariant $\alp(V)=c_*^F(V)$ for $V \in ob(\m V^{emb}_k)$ also multiplicative
(as follows from \cite[Example 4.2.5]{Fulton-book}).\\

\noindent
{\bf Acknowlwdgements.} We would like to thank Paolo Aluffi and Markus Banagl for useful discussions, and also Yuli B. Rudyak for providing the precise reference for Theorem \ref{rudyak}. 
Finally we would like to thank L\^e D\~ung Tr\'ang for useful suggestions and the referee for his/her careful reading and valuable suggestions and comments.

\end{document}